\documentclass[reqno]{amsart}
\usepackage[bookmarks=false]{hyperref} 
\usepackage{amssymb}
\usepackage{amsmath}
\usepackage{amsthm} 
\usepackage{mathtools}
\usepackage{color} 
\usepackage[svgnames]{xcolor}
\usepackage[utf8]{inputenc}
\usepackage[T1]{fontenc}
\usepackage{graphicx}
\usepackage{placeins}
\usepackage{vmargin}
\usepackage{setspace}
\usepackage{bbm}

\hypersetup{
  colorlinks=true,
  linkcolor=Blue  ,          
  citecolor=Red,        
  filecolor=Magenta,      
  urlcolor=Green,           
pdfpagemode=UseOutlines,
pdftitle={},
pdfauthor={Tin Van Phan <van-tin.phan@univ-tlse3.fr>},
pdfsubject={On the Cauchy problem for
  a derivative nonlinear Schrodinger equation with nonvanishing boundary conditions},
pdfkeywords={Cauchy problem,derivative nonlinear Schr\"odinger equation,nonvanishing boundary conditions} 
}

\numberwithin{equation}{section}
\numberwithin{figure}{section}

\newtheorem{theorem}{Theorem}[section]
\newtheorem{proposition}[theorem]{Proposition}
\newtheorem{lemma}[theorem]{Lemma}

\newtheorem{corollary}[theorem]{Corollary}

\theoremstyle{definition}
\newtheorem{definition}[theorem]{Definition}

\theoremstyle{remark}
\newtheorem{remark}[theorem]{Remark}

\DeclarePairedDelimiter{\norm}{\lVert}{\rVert}

\newcommand{\N}{\mathbb{N}}
\newcommand{\R}{\mathbb{R}}
\newcommand{\C}{\mathbb C}

\renewcommand{\leq}{\leqslant}
\renewcommand{\geq}{\geqslant}

\DeclareMathAlphabet{\mathpzc}{OT1}{pzc}{m}{it}
\renewcommand{\Re}{\mathcal R\!\mathpzc{e}}
\renewcommand{\Im}{\mathcal I\!\mathpzc{m}}

\begin{document}

\title[Instability of multi-solitons]{Instability of multi-solitons for derivative nonlinear Schr\"odinger equations}

\author[Phan Van Tin]{Phan Van Tin}

\address[Phan Van Tin]{Institut de Math\'ematiques de Toulouse ; UMR5219,
  \newline\indent
  Universit\'e de Toulouse ; CNRS,
  \newline\indent
  UPS IMT, F-31062 Toulouse Cedex 9,
  \newline\indent
  France}
\email[Phan Van Tin]{van-tin.phan@univ-tlse3.fr}


\date{\today}

\begin{abstract}
In \cite{CoWu18} and \cite{TaXu18}, the authors proved the stability of multi-solitons for derivative nonlinear Schr\"odinger equations. Roughly speaking, sum of finite stable solitons is stable. We predict that if there is one unstable solition then multi-soliton is unstable. This prediction is proved in \cite{CoLe11} for classical nonlinear Schr\"odinger equations. In this paper, we proved this prediction for derivative nonlinear Schr\"odinger equations by using the method of C\^{o}te-Le Coz \cite{CoLe11} with the help of Gauge transformation.
\end{abstract}

\maketitle
\tableofcontents

\section{Introduction}

We consider the following derivative nonlinear Schr\"odinger equations
\begin{equation}
\label{eq 1}
iu_t+u_{xx}+i|u|^2u_x+b|u|^4u=0,
\end{equation}
and 
\begin{equation}
\label{eq2}
iu_t+u_{xx}+i|u|^{2\sigma}u_x=0,
\end{equation}
where $b\in \R$, $\sigma\geq 1$ and $u: \R\times\R\rightarrow \C$ is unknown function. 

The local well posedness and the global well posedness of derivative nonlinear Schr\"odinger equations were studied in many works (see \cite{BaPe20,BiLi01,CoKeStTaTa01,CoKeStTaTa02,GuWu95,HaOz92,
HaOz94,JeLiPeSu20,MiWuXu11,Oz96,Ta99,TsFu80,TsFu81,Wu15,Wu13} for \eqref{eq 1} and see e.g \cite{HaOz16,Santos15} for \eqref{eq2}). The existence of blow up solutions for \eqref{eq 1} and \eqref{eq2} is an open question. 

The equation \eqref{eq 1} and \eqref{eq2} have Hamilton structures and they do not possess the Galilean invariant. The family of solitons of the derivative nonlinear Schr\"odinger equations have two parameters. A soliton of \eqref{eq 1} and \eqref{eq2} is a solution of form $R_{\omega,c}(t,x)=e^{i\omega t}\phi_{\omega,c}(x-ct)$, where $\omega >0$ and $c^2<4\omega$. The stability and instability of solitons are proved in many works (see \cite{Ohta14,KwWu18,Ha21,GuWu95,CoOh06} for \eqref{eq 1} and \cite{LiSiSu13,GuNiWu20} for \eqref{eq2}). 

Multi-soliton is a solution of \eqref{eq 1}, \eqref{eq2} which behaves at large time like a sum of finite solitons. In \cite{Tinpaper3,Tinpaper4}, Tin proved the existence of multi-solitons for \eqref{eq 1} and \eqref{eq2} respectively. The author used fixed point method, Strichartz estimates and gauge transformations to obtain the desired results. The stability of multi-solitons was proved in \cite{CoWu18} for \eqref{eq 1} in the case $b=0$ and for \eqref{eq2} in the case $\sigma \in (1,2)$ in \cite{TaXu18} provided all solitons are stable. Roughly speaking, the multi-solitons behave at large time like a sum of stable solitons are stable. We predict that if there is one unstable soliton then multi-soliton is unstable in some sense. This prediction was proved in the case of classical nonlinear Schr\"odinger equation by the work of C\^{o}te-Le Coz \cite{CoLe11}. In this paper, using the idea of C\^{o}te-Le Coz, we show that if soliton of \eqref{eq 1} (\eqref{eq2}) is linearly unstable then it is orbitally unstable. Moreover, multi-soliton behaving like a sum of one unstable soliton and finite solitons is not unique and unstable. 

\subsection{Instability of multi-solitons for \eqref{eq 1}}
\label{sec1}
The flow of \eqref{eq 1} in $H^1(\R)$ satisfies the following conservation laws.
\begin{align*}
\text{Energy} &\quad E(u):=\frac{1}{2}\norm{u_x}^2_{L^2}+\frac{1}{4}\Im\int_{\R}|u|^2u_x\overline{u}\,dx-\frac{b}{6}\norm{u}^6_{L^6},\\
\text{Mass} &\quad Q(u):=\frac{1}{2}\norm{u}^2_{L^2},\\
\text{Momentum}&\quad P(u):=-\frac{1}{2}\Im\int_{\R}u_x\overline{u}\,dx.
\end{align*}
For each $\omega,c\in\R$ and $u\in H^1(\R)$, we define
\[
S_{\omega,c}(u)=E(u)+\omega Q(u)+cP(u).
\]
Recall that a soliton of \eqref{eq 1} is a solution of form $R_{\omega,c}=e^{i\omega t}\phi_{\omega,c}(x-ct)$, for $\phi_{\omega,c}$ is a critical point of $S_{\omega,c}$. Moreover, $\phi_{\omega,c}$ is (up to phase shift and translation) of form 
\[
\phi_{\omega,c}=\Phi_{\omega,c}\exp\left(\frac{ic}{2}x-\frac{i}{4}\int_{-\infty}^x|\Phi_{\omega,c}(y)|^2\,dy\right),
\]
where $\Phi_{\omega,c}$ is given by if $\gamma:=1+\frac{16}{3}b>0$,
\begin{equation}\label{formula Phi}
\Phi_{\omega,c}^2(x)=\left\{ \begin{matrix}
\frac{2(4\omega-c^2)}{\sqrt{c^2+\gamma (4\omega-c^2)}\cosh(\sqrt{4\omega-c^2} x)-c} & \text{ if } -2\sqrt{\omega}<c<2\sqrt{\omega},\\
\frac{4c}{(cx)^2+\gamma} & \text{ if } c=2\sqrt{\omega},
\end{matrix}
\right.
\end{equation}
and if $\gamma \leq 0$ ($b \leq -\frac{3}{16}$),
\[
\Phi_{\omega,c}^2(x)=\frac{2(4\omega-c^2)}{\sqrt{c^2+\gamma(4\omega-c^2)}\cosh(\sqrt{4\omega-c^2}x)-c} \text{ if } -2\sqrt{\omega}<c<-2s_{*}\sqrt{\omega},
\]
where $s_{*}=s_{*}(\gamma)=\sqrt{\frac{-\gamma}{1-\gamma}}$. 

We note that the following condition on the parameters $\gamma$ and $(\omega,c)$ is a necessary and sufficient condition for the existence of non-trivial solutions of \eqref{eq 1} vanishing at infinity (see \cite{BeLi831}):
\begin{align*}
\text{if } \gamma>0 (\Leftrightarrow b>\frac{-3}{16}) \text{ then } & -2\sqrt{\omega}<c\leq 2\sqrt{\omega},\\
\text{if } \gamma\leq 0 (\Leftrightarrow b\leq\frac{-3}{16}) \text{ then }& -2\sqrt{\omega}<c< -2s_{*}\sqrt{\omega}.
\end{align*}

Define $d(\omega,c)=S_{\omega,c}(\phi_{\omega,c})$ and 
\begin{align*}
H_{\omega,c}(v)&=(E''(\phi_{\omega,c})+\omega Q''(\phi_{\omega,c})+cP''(\phi_{\omega,c}))(v)\\
&=-\partial_{xx}v+\omega v+ic\partial_x v-2i\partial_x\phi_{\omega,c}\Re(\phi_{\omega,c}\overline{v})-i|\phi_{\omega,c}|\partial_x v\\
&\quad -b(|\phi_{\omega,c}|^4v+4|\phi_{\omega,c}|^2\phi_{\omega,c}\Re(\phi_{\omega,c}\overline{v})).
\end{align*}
Let $n(H_{\omega,c})$ be the number of negative eigenvalue of $H_{\omega,c}$ and $p(d''(\omega,c))$ be the number of positive eigenvalue of the matrix $d''(\omega,c)$, which is defined by
\[
d''(\omega,c)=\begin{bmatrix}
\partial_{\omega}^2d(\omega,c)&\partial_c\partial_{\omega}d(\omega,c)\\
\partial_{\omega}\partial_cd(\omega,c)& \partial_c^2d(\omega,c)
\end{bmatrix}=\begin{bmatrix}
\partial_{\omega}Q(\phi_{\omega,c})&\partial_{c}Q(\phi_{\omega},c)\\
\partial_{\omega}P(\phi_{\omega,c})&\partial_cP(\phi_{\omega,c})
\end{bmatrix}.
\]

The stability/instability of solitons $R_{\omega,c}$ can be given by the abstract theory of Grillakis-Shatah-Strauss \cite{GrShSt87,GrShSt90}. We have the following result.
\begin{theorem}\label{GrShSt}
\[
p(d''(\omega,c))\leq n(H_{\omega,c}). 
\]
Furthermore, under the condition that $d$ is non-degenerate at $(\omega,c)$:
\begin{itemize}
\item[(i)] If $p(d''(\omega,c))=n(H_{\omega,c})$ then $R_{\omega,c}$ is orbitally stable;
\item[(ii)] If $n(H_{\omega,c})-p(d''(\omega,c))$ is odd then $R_{\omega,c}$ is orbitally unstable.
\end{itemize}
\end{theorem}

Let $K \in\N$, $K>1$. For each $1 \leq j \leq K$, let $(\theta_j,x_j)\in\R^2$ and $(c_j,\omega_j)$ satisfy the condition of existence of soliton. For each $j \in \{1,2,..,K\}$, we set
\begin{equation}\label{define of R_j}
R_j(t,x)=e^{i\theta_j}R_{\omega_j,c_j}(t,x-x_j).
\end{equation}
We define for each $j$, $h_j=\sqrt{4\omega_j-c_j^2}$. As in \cite[Lemma 4.1]{Tinpaper3}, 
\begin{equation}\label{estimate R_j}
|R_j(t,x)| \lesssim e^{-\frac{h_j}{2}|x-c_jt|}.
\end{equation}
The profile of a multi-soliton is a sum of the form:
\begin{equation}\label{profileofinfinitesoliton}
R =\sum_{j=1}^{K}R_j.
\end{equation}
A solution of \eqref{eq 1} is called a multi-soliton if
\[
\norm{u(t)-R(t)}_{H^1} \rightarrow 0 \text{ as } t \rightarrow \infty.
\]
 
Since solutions of \eqref{eq 1} are invariant by phase shift and translation, we may assume that $\theta_1=x_1=0$ without loss of generality. For convenience, we denote $\phi_j=\phi_{\omega_j,c_j}$ for all $j$ and $\phi=\phi_1$. Then $R_1(t,x)=e^{i\omega_1 t}\phi(x-c_1t)$. We have
\begin{equation}
\label{equation of phi}
-\phi_{xx}+\omega_1\phi+ic_1\phi_x-i|\phi|^2\phi_x-b|\phi|^4\phi=0.
\end{equation} 
Let $u(t,x)$ be a solution of \eqref{eq 1} and set $u= e^{i\omega_1 t} (\phi(x-c_1 t)+v(t,x-c_1 t))$. Using \eqref{equation of phi}, we have 
\begin{align*}
0&=iu_t+u_{xx}+i|u|^2u_x+b|u|^4u\\
&=i(i\omega_1 e^{i\omega_1 t} (\phi+v)+e^{i\omega_1 t}(-c_1\phi_x+v_t-c_1v_x))+e^{i\omega_1 t}(\phi_{xx}+v_{xx})\\
&\quad +e^{i\omega_1 t}i|\phi+v|^2(\phi_x+v_x)+be^{i\omega_1 t}|\phi+v|^4(\phi+v)\\
&=e^{i\omega_1 t}(-\omega_1 v+iv_t-ic_1v_x+v_{xx}+i(|\phi+v|^2(\phi_x+v_x)-|\phi|^2\phi_x)+b(|\phi+v|^4(\phi+v)-|\phi|^4\phi))\\
&=e^{i\omega t}i(v_t+L_{\C}(v)+\mathcal{M}_{\C}(v)), 
\end{align*}
where $L_{\C}$ is linearized operator around $R_1$ and is defined by
\begin{equation}\label{define of L_C}
L_{\C}(v)=-iv_{xx}+i\omega_1 v-c_1v_x+2\Re(\phi\overline{v})\phi_x+|\phi|^2v_x-ib(|\phi|^4v+4|\phi|^2\phi\Re(\phi\overline{v})),
\end{equation}
and the quadratic term in $v$, $\mathcal{M}_{\C}$ is defined by
\[
\mathcal{M}_{\C}(v)=2\Re(\phi \overline{v})v_x+|v|^2\phi_x+|v|^2v_x-ib(\phi+v)(4\Re(\phi\overline{v})^2+2|v|^2|\phi|^2+4|v|^2\Re(\phi\overline{v})+|v|^4).
\]
We may check that $L_{\C}=iH_{\omega_1,c_1}$. We need the following assumption.
\begin{equation}
\tag{A1}
L_{\C} \text{has an eigenvalue } \lambda\in\C \text{ such that } \rho:=\Re\lambda>0. 
\label{A1}
\end{equation}
Our first goal is to prove that if soliton is linearly unstable then it is orbitally unstable. To do this, we prove the following result.
\begin{theorem}\label{thm1}
Assume that \eqref{A1} holds. Then there exists a function $Y(t)$ such that $\norm{Y(t)}_{H^2} \leq Ce^{-\rho t}$ and $e^{\rho t}\norm{Y(t)}_{H^2}$ is non-zero and periodic (where $\rho$ is given by \eqref{A1}) and $Y(t)$ is a solution to the linearized flow around $R_1$. For all $a\in\R$, there exist $T_0\in\R$ large enough, a constant $C>0$ and a solution $u_a$ to \eqref{eq 1} defined on $[T_0,\infty)$ such that
\[
\norm{u_a(t)-R_1(t)-aY(t)}_{H^2} \leq Ce^{-2\rho t}, \quad \forall t\geq T_0.
\]
\end{theorem}
As a consequence of \ref{thm1}, we prove that under \eqref{A1}, $R_1$ is orbitally unstable. We prove the following result.
\begin{corollary}\label{corollary1}
Under the hypothesis of Theorem \ref{thm1}, $R_1$ is orbitally unstable in the following sense. There exist $\varepsilon>0$, $(T_n)\subset \R^-$, $(u_{0,n}) \subset H^2(\R)$ and solution $(u_n)$ of \eqref{A1} defined on $[T_n,0]$ with $u_n(0)=u_{0,n}$ such that 
\[
\lim_{n\rightarrow\infty}\norm{u_{0,n}-R_1(0)}_{H^2}=0 \quad\text{ and }\quad \inf_{y\in\R,\theta\in\R}\norm{u_n(T_n)-e^{i\theta}\phi(\cdot-y)}_{L^2}\geq \varepsilon \quad \forall n\in\N.
\]
\end{corollary}
Under the assumption of Theorem \ref{thm1}, we prove the existence of a one parameter family of multi-solitons. This implies that multi-soliton is not unique. Moreover, we prove instability for high relative speed of multi-solitons.

\begin{theorem}\label{thm2}
Let $K\in \N$, $K >1$. For each $j=1,...,K$, let $(\theta_j,x_j)\in\R^2$ and $(c_j,\omega_j)$ satisfy the condition of existence of soliton and $R_j$ be defined by \eqref{define of R_j}. Let $h_{*}=\inf{\inf_{j}h_j,2\alpha}$, where $\alpha$ is the given constant in Proposition \ref{pro25} and $v_{*}=\frac{1}{9}\min\{|c_j-c_k|:j,k=1,...,K,\ j\neq k\}$. Assume that \eqref{A1} holds. There exists $v_{\natural}>0$ such that if $v_{*}>v_{\natural}$ then the following holds. \\
There exist $Y(t)$ such that $\norm{Y(t)}_{H^2}\leq Ce^{-\rho t}$ and $e^{\rho t}\norm{Y(t)}_{H^2}$ is non-zero and periodic, where $\rho$ is given by \eqref{A1} and $Y(t)$ is a solution to the linearized flow around $R_1$. For all $a\in\R$, there exist $T_0\in\R$ large enough, a solution $u_a$ to \eqref{eq 1}, and a constant $C>0$ such that
\[
\left\lVert u_a(t)-\sum_{j=1}^KR_j(t)-aY(t)\right\rVert_{H^2}\leq Ce^{-2\rho t}.
\]  
\end{theorem}

\begin{corollary}\label{corollary2}
Let $R$ be the multi-soliton profile defined by \eqref{profileofinfinitesoliton}. Under the hypotheses of Theorem \ref{thm2}, the multi-soliton around $R$ satisfies the following instability property. There exists $\varepsilon>0$, such that for all $n\in\N\setminus\{0\}$ and for $T>0$ large enough the following holds. There exist $I_n,J_n\in\R$, $J_n<I_n<-T$ and a solution $w_n\in C([J_n,I_n],H^2(R))$ to \eqref{eq 1} such that
\[
\lim_{n\rightarrow\infty}\norm{w_n(I_n)-R(I_n)}_{H^2}=0, \quad \text{ and } \inf_{\begin{matrix}y_j\in\R,\theta_j\in\R\\ j=1,...,K\end{matrix}}\left\lVert w_n(J_n)-\sum_{j=1}^K\phi_{j}(\cdot-y_j)e^{i\theta_j} \right\rVert_{L^2}\geq \varepsilon.
\] 
\end{corollary}
\begin{remark}
Replacing $\phi$ by $\overline{\phi}$ in the definition of $L_{\C}$, we obtain the new operator denoted by $L_{\C}^{\overline{\phi}}$. By similar argument in \cite[Proof of Corollary 2]{CoLe11}, we may prove that if $L_{\C}^{\overline{\phi}}$ has a eigenvalue with positive real part then the soliton $R_1$ (in Corollary \ref{corollary1}) and the multi-soliton $R$ (in Corollary \ref{corollary2}) are unstable forward in time. However, not like in \cite{CoLe11} for classical nonlinear Schr\"odinger equation, in our case \eqref{A1} does not imply that $L_{\C}^{\overline{\phi}}$ has an eigenvalue with positive real part.
\end{remark}

From \cite[Theorem 5.1]{GrShSt90}, if $d''(\omega_1,c_1)$ is non-singular and $n(H_{\omega_1,c_1})-p(d''(\omega_1,c_1))$ odd then $-iH_{\omega_1,c_1}$ has at least one pair of real non-zero eigenvalues $\pm\lambda$. In that case, \eqref{A1} holds. We have the following result. 
\begin{theorem}
Assume that $d''(\omega_1,c_1)$ is non-singular and $n(H_{\omega_1,c_1})-p(d''(\omega_1,c_1))$ odd. Then the conclusions of Corollary \ref{corollary1} and Corollary \ref{corollary2} hold. 
\end{theorem}

\begin{remark}
From the works of Colin-Ohta \cite{CoOh06}, Ohta \cite{Ohta14} and Hauashi \cite{Ha21}, we see that $p(d''(\omega,c))=1$ if $b=0$ or $b<0$ or $b>0$ and $-2\sqrt{\omega}<c<2\kappa \sqrt{\omega}$; $p(d''(\omega,c))=0$ if $b>0$ and $2\kappa \sqrt{\omega}<c<2\sqrt{\omega}$ for some constant $\kappa=\kappa(b)\in (0,1)$. We predict that $n(H_{\omega,c})=1$ for all $b$ and the condition $n(H_{\omega_1,c_1})-p(d''(\omega_1,c_1))$ odd is replaced by $p(d''(\omega_1,c_1))=0$. 
\end{remark}

\subsection{Instability of multi-solitons for \eqref{eq2}}
\label{sec2}
In this section, for simplicity, we use the same notation in Section \ref{sec1}. 

The flow of \eqref{eq2} in $H^1(\R)$ satisfies the following conservation laws.
\begin{align*}
\text{Energy} &\quad E(u):=\frac{1}{2}\norm{u_x}^2_{L^2}+\frac{1}{2(\sigma+1)}\Im\int_{\R}|u|^{2\sigma}u_x\overline{u}\,dx,\\
\text{Mass} &\quad Q(u):=\frac{1}{2}\norm{u}^2_{L^2},\\
\text{Momentum}&\quad P(u):=-\frac{1}{2}\Im\int_{\R}u_x\overline{u}\,dx.
\end{align*} 
For each $\omega,c\in\R$ and $u\in H^1(\R)$, we define
\[
S_{\omega,c}(u)=E(u)+\omega Q(u)+cP(u).
\]
A soliton of \eqref{eq2} is a solution of form $R_{\omega,c}(t,x)=e^{i\omega t}\phi_{\omega,c}(x-ct)$, for $\phi_{\omega,c}$ is a critical point of $S_{\omega,c}$. Moreover, $\phi_{\omega,c}$ is (up to phase shift and translation) of form
\[
\phi_{\omega,c}(x)=\Phi_{\omega}(x)\exp\left(\frac{ic}{2}x-\frac{1}{2\sigma+2}\int_{-\infty}^x\Phi_{\omega,c}^{2\sigma}(y)\,dy\right),
\]
where $\omega>\frac{c^2}{4}$ and
\[
\Phi_{\omega,c}^{2\sigma}(y)=\frac{(\sigma+1)(4\omega-c^2)}{2\sqrt{\omega}\left(\cosh(\sigma\sqrt{4\omega-c^2} y)-\frac{c}{2\sqrt{\omega}}\right)}.
\]
For each $\omega,c\in\R$, let $d(\omega,c),H_{\omega,c},n(H_{\omega,c}),p(d''(\omega,c))$ be defined as in Section \ref{sec1}. Similar in the case \eqref{eq 1}, Stability/instability of solitons of \eqref{A2} obeys Theorem \ref{GrShSt}. In \cite{LiSiSu13}, the authors proved that $n(H_{\omega,c})=1$ for all $\sigma>0$. Thus, $R_{\omega,c}$ is orbitally stable if $p(d''(\omega,c))=1$ and orbitally unstable if $p(d''(\omega,c))=0$. 

Let $K \in\N$, $K>1$. For each $1 \leq j \leq K$, let $(\theta_j,x_j)\in\R^2$ and $(c_j,\omega_j)$ satisfy the condition of existence of soliton. For each $j \in \{1,2,..,K\}$, we set
\begin{equation*}
R_j(t,x)=e^{i\theta_j}R_{\omega_j,c_j}(t,x-x_j).
\end{equation*}
We define for each $j$, $h_j=\sqrt{4\omega_j-c_j^2}$. As in \cite[Lemma 3.1]{Tinpaper4}, 
\[
|R_j(t,x)| \lesssim e^{-\frac{h_j}{2}|x-c_jt|}.
\]
The profile of a multi-soliton is a sum of the form:
\begin{equation}\label{profile multisoliton}
R =\sum_{j=1}^{K}R_j.
\end{equation}
Since \eqref{eq2} is invariant under phase shift and translation, we may assume that $\theta_1=x_1=0$. For convenience, we denote $\phi_{j}=\phi_{\omega_j,c_j}$ and $\phi=\phi_1$. By an elementary calculation, we see that the linearized operator around $R_1$ of \eqref{eq2} is the following.
\[
L_{\C}(v)=-iv_{xx}+i\omega_1 v-c_1v_x+v_x|\phi|^{2\sigma}+2\sigma\phi_x\Re(\phi\overline{v})|\phi|^{2(\sigma-1)}.
\]
We may check that $L_{\C}=iH_{\omega_1,c_1}$. We need the following assumption.
\begin{equation}
\tag{A2}
L_{\C} \text{has an eigenvalue } \lambda\in\C \text{ such that } \rho:=\Re\lambda>0. 
\label{A2}
\end{equation}
We have the following result.
\begin{theorem}\label{thm3}
Let $\sigma=1$ or $\sigma=2$ or $\sigma\geq \frac{5}{2}$. Under \eqref{A2}, $R_1$ is orbitally unstable by the same sense as in Corollary \ref{corollary1}.
\end{theorem}
Moreover, we have the following result.
\begin{theorem}\label{thm4}
Let $\sigma=1$ or $\sigma=2$ or $\sigma\geq \frac{5}{2}$. Let $R$ be the multi-solitons profile defined by \eqref{profile multisoliton}. Assume that \eqref{A2} holds. Then the multi-soliton around $R$ is unstable by the same sense as in Corollary \ref{corollary2}.
\end{theorem}
Using \cite[Theorem 5.1]{GrShSt90}, we have if $d''(\omega_1,c_1)$ is non-singular and $p(d''(\omega,c))=0$ then $-iH_{\omega_1,c_1}$ has onr pair of real non-zero eigenvalue $\pm\lambda$. In that case, \eqref{A2} holds. Thus, we have the following result.*
\begin{theorem}
Let $\sigma=1$ or $\sigma=2$ or $\sigma\geq \frac{5}{2}$ and $R_1$ be such that $d''(\omega_1,c_1)$ is non-singular and $p(d''(\omega_1,c_1))=0$ then the conclusions of Theorem \ref{thm3} and Theorem \ref{thm4} hold.
\end{theorem}

\begin{remark}
Define 
\begin{align*}
\varphi(t,x)&=\exp\left(\frac{i}{2}\int_{-\infty}^x|u(t,y)|^2\,dy\right)u(t,x),\\
\psi&=\partial_x\varphi-\frac{i}{2}|\varphi|^{2\sigma}\varphi.
\end{align*}
From \cite[page 6]{Tinpaper4}, if $u$ solves \eqref{eq2} then $(\varphi,\psi)$ solves
\begin{align*}
L\varphi&=P(\varphi,\psi),\\
L\psi&=Q(\varphi,\psi),
\end{align*}
where $P,Q$ are defined by
\begin{align*}
P(\varphi,\psi)&=i\sigma |\varphi|^{2(\sigma-1)}\varphi^2\overline{\psi}-\sigma(\sigma-1)\varphi\int_{-\infty}^x|\varphi|^{2(\sigma-2)}\Im(\psi^2\overline{\varphi}^2)\,dy,\\
Q(\varphi,\psi)&=-i\sigma|\varphi|^{2(\sigma-1)}\psi^2\overline{\varphi}-\sigma(\sigma-1)\psi\int_{-\infty}^x|\varphi|^{2(\sigma-2)}\Im(\psi^2\overline{\varphi}^2)\,dy.
\end{align*}
Since \cite[Remark 1.2]{Tinpaper4}, the conditions $\sigma=1$ or $\sigma=2$ or $\sigma\geq \frac{5}{2}$ ensure that $P(\varphi,\psi)$ and $Q(\varphi,\psi)$ are Lipschitz continuous on bounded set of $H^1(\R) \times H^1(\R)$. This is important point in the proof of Theorem \ref{thm3}, \ref{thm4}.
\end{remark}

\begin{remark}
From the work of Liu-Simpson-Sulem \cite{LiSiSu13}, we have if $\sigma\geq 2$ or $\sigma \in (1,2)$ and $2z_0\sqrt{\omega}<c<2\sqrt{\omega}$ then $p(d''(\omega,c))=0$ and if $\sigma\in (0,1)$ or $\sigma\in (1,2)$ and $-2\sqrt{\omega}<c<2z_0\sqrt{\omega}$ then $p(d''(\omega,c))=1$.
\end{remark}
The proofs of Theorem \ref{thm3} and \ref{thm4} are similar the proofs of Corollary \ref{corollary1} and Corollary \ref{corollary2} respectively. In this paper, we admit this and we only focus on the proofs of the results in Section \ref{sec1}.


\section{Proof of main results}

As said above, we only prove the results in Section \ref{sec1}. The results in Section \ref{sec2} are proved by similar argument. 

\subsection{Construction of approximation profiles}
For convenience, we use the same notation as in \cite{CoLe11}. We identify $\C$ with $\R^2$ and use the notation $a+ib=\left(\begin{matrix} a \\ b \end{matrix}\right)$ $(a,b\in\R)$. Given $v\in\C$, we denote $v^+$ is its real part and $v^-$ is its imaginary part. To avoid confusion, we denote with an index whether we consider the operator with $\C$-, $\R^2$-, or $\C^2$-valued functions. 

Let $L_{\C}$ be defined by \eqref{define of L_C}. We define 
\begin{align*}
\mathcal{L}_{\C}(v)&=-iv_{xx}+2\Re(R_1 \overline{v})R_{1x}+|R_1|^2v_x\\
&\quad -ib(|R_1|^4v+4|R|^2R\Re(R\overline{v})),
\end{align*}
and the nonlinear operators
\begin{align*}
\mathcal{N}_{\C}(v)&=2\Re(R_1\overline{v})v_x+|v|^2R_{1x}+|v|^2v_x\\
&-ib\left((R_1+v)(4\Re(R_1\overline{v})^2+2|v|^2|R_1|^2+4|v|^2\Re(R_1\overline{v}+|v|^4)+4|R_1|^2v \Re(R_1\overline{v})\right)\\
\mathcal{M}_{\C}(v)&=e^{-i\omega_1 t}\mathcal{N}_{\C}(e^{i\omega_1 t}v)=2\Re(\phi\overline{v})v_x+|v|^2\phi_x+|v|^2v_x\\
&\quad -ib(\phi+v)(4\Re(\phi\overline{v})^2+2|v|^2|\phi|^2+4|v|^2\Re(\phi\overline{v})+|v|^4)-4ib|\phi|^2v\Re(\phi\overline{v}).
\end{align*}
We have 
\begin{align*}
&L_{\R^2}\begin{pmatrix}v^+\\v^-\end{pmatrix}\\
&=\begin{pmatrix} \Re(L_{\C}(v))\\ \Im(L_{\C}(v))\end{pmatrix}\\
&=\begin{pmatrix}
v^-_{xx}-\omega_1v^--c_1v^+_x+2(v^+\phi^+ +v^-\phi^-)\phi^+_x+|\phi|^2v_x^++b|\phi|^4v^-+4b|\phi|^2\phi^-(\phi^+v^+ + \phi^-v^-)\\
-v^+_{xx}+\omega_1v^+-c_1v^-_x+2(v^+\phi^+ + v^-\phi^-)\phi^-_x+|\phi|^2v^-_x-b|\phi|^4v^+-4b|\phi|^2\phi^+(\phi^+v^+ +\phi^-v^-)
\end{pmatrix} \\
&=
\begin{pmatrix}
-c_1\partial_x+2\phi^+_x\phi^- +|\phi|^2\partial_x+4b|\phi|^2\phi^+\phi^- & \partial_{xx}-\omega_1+2\phi^-\phi_x^++b|\phi|^4+4b|\phi|^2(\phi^-)^2\\ -\partial_{xx}+\omega_1+2\phi_x^-\phi^+-b|\phi|^4-4b|\phi|^2(\phi^+)^2& -c_1\partial_x+2\phi^-\phi_x^- +|\phi|^2\partial_x-4b|\phi|^2\phi^+\phi^-
\end{pmatrix} \begin{pmatrix}
v^+\\v^-
\end{pmatrix}
\end{align*}
We see that $L_{\R^2}$ is an $R$-linear operator on $H^2(\R,R^2)\rightarrow L^2(\R,\R^2)$. To have some eigenfunctions, we extend $L_{\R^2}$ to $L_{\C^2}: H^2(\R,\C^2)\rightarrow L^2(\R,\C^2)$, which is a $\C$-linear operator. 

Define $v=e^{\frac{ic_1}{2}x}\tilde{v}$ and $\phi=e^{\frac{ic_1}{2}x}\tilde{\phi}$. By an elementary calculation, we have
\begin{align*}
L_{\C}(v)&=L_{\C}(e^{\frac{ic_1}{2}x}\tilde{v})\\
&=e^{\frac{ic_1}{2}x}L_{\C}^{\tilde{\phi}}(\tilde{v}), 
\end{align*}
where
\begin{align*}
L_{\C}^{\tilde{\phi}}(\tilde{v})&=-i\tilde{v}_{xx}+i\left(\omega_1-\frac{c_1^2}{4}\right)\tilde{v}+2\Re(\tilde{\phi}\overline{\tilde{v}})\left(\frac{ic_1}{2}\tilde{\phi}-\tilde{\phi}_x-2ib|\tilde{\phi}|^2\tilde{\phi}\right)\\
&\quad +|\tilde{\phi}|^2\left(\frac{ic_1}{2}+\tilde{v}_x\right)-ib|\tilde{\phi}|^4\tilde{v}.
\end{align*}  
Thus, $L_{\C^2}^{\tilde{\phi}}$ equals to
\begin{align*}
\begin{pmatrix}
W_{1,1}+W_{1,2}\partial_x & \partial_{xx}-\left(w_1-\frac{c_1^2}{4}\right)+W_2 \\
-\partial_{xx}+\left(\omega_1-\frac{c_1^2}{4}\right)+W_3& W_{4,1}+W_{4,2}\partial_x,
\end{pmatrix}.
\end{align*}
where
\begin{align*}
W_{1,1}&=\tilde{\phi}^+\left(\frac{-c_1}{2}\tilde{\phi}^- +\tilde{\phi}_x^+ + 2b|\tilde{\phi}|^2\tilde{\phi}^-\right),\\
W_{1,2}&=|\tilde{\phi}|^2,\\
W_2&=\tilde{\phi}^-\left(-\frac{c_1}{2}\tilde{\phi}^- + \tilde{\phi}_x^++2b|\tilde{\phi}|^2\tilde{\phi}^-\right)-\frac{c_1}{2}|\tilde{\phi}|^2+b|\tilde{\phi}|^4\\
W_3&=\tilde{\phi}^+ \left(\frac{c_1}{2}\tilde{\phi}^+ + \tilde{\phi}_x^+ -2b|\tilde{\phi}|^2\tilde{\phi}^+\right)+\frac{c_1}{2}|\tilde{\phi}|^2-b|\tilde{\phi}|^4\\
W_{4,1}&=\tilde{\phi}^-\left(\frac{c_1}{2}\tilde{\phi}^+ + \partial_x\tilde{\phi}^- -2b|\tilde{\phi}|^2\tilde{\phi}^+\right)\\
W_{4,2}&=|\tilde{\phi}|^2.
\end{align*}
Thus, $W_{1,1},W_{1,2},W_2,W_3,W_{4,1},W_{4,2}$ are exponentially decaying at infinity.  \\
Moreover,
\begin{align*}
L_{\R^2}^{\tilde{\phi}}\begin{pmatrix}
\tilde{v}^+\\ \tilde{v}^- \end{pmatrix} 
&=\begin{pmatrix}
\Re(L_{\C}^{\tilde{\phi}}(\tilde{v}))\\ \Im(L_{\C}^{\tilde{\phi}}(\tilde{v}))\end{pmatrix}
=\begin{pmatrix}
\Re(e^{-\frac{ic_1}{2}x}L_{\C}^{\phi}(v))\\
\Im(e^{-\frac{ic_1}{2}x}L_{\C}^{\phi}(v))
\end{pmatrix}\\
&=\begin{pmatrix}
\cos\left(\frac{c_1}{2}x\right) & \sin\left(\frac{c_1}{2}x\right)\\
-\sin\left(\frac{c_1}{2}x\right)& \cos\left(\frac{c_1}{2}x\right) 
\end{pmatrix} \begin{pmatrix}
\Re(L_{\C}^{\phi}(v))\\ \Im(L_{\C}^{\phi}(v))
\end{pmatrix}\\
&=\begin{pmatrix}
\cos\left(\frac{c_1}{2}x\right) & \sin\left(\frac{c_1}{2}x\right)\\
-\sin\left(\frac{c_1}{2}x\right)& \cos\left(\frac{c_1}{2}x\right) 
\end{pmatrix} L_{\R^2}\begin{pmatrix}
v^+\\ v^-
\end{pmatrix}\\
&=\begin{pmatrix}
\cos\left(\frac{c_1}{2}x\right) & \sin\left(\frac{c_1}{2}x\right)\\
-\sin\left(\frac{c_1}{2}x\right)& \cos\left(\frac{c_1}{2}x\right) 
\end{pmatrix} L_{\R^2}\begin{pmatrix}
\cos\left(\frac{c_1}{2}x\right) & \sin\left(\frac{c_1}{2}x\right)\\
-\sin\left(\frac{c_1}{2}x\right)& \cos\left(\frac{c_1}{2}x\right) 
\end{pmatrix}^{-1}\begin{pmatrix}
\tilde{v}^+\\ \tilde{v}^-
\end{pmatrix}.
\end{align*}
This implies that the spectrum set and the resolvent set of $L_{\C}$ ($L_{\C^2}$) are same to the spectrum set and the resolvent set of $L_{\C}^{\tilde{\phi}}$ ($L_{\C^2}^{\tilde{\phi}}$). 

Let $\alpha>0$ be the decay rate given by Proposition \ref{pro25} for eigenfunctions of $L$ with eigenvalue $\lambda$ (see \eqref{A1}). Taking a small value of $\alpha$, we assume that $\alpha\in \left(0,\frac{h_1}{2}\right)$, where $h_1=\sqrt{4\omega_1-c_1^2}$. For $\mathbb{K}=\R,\R^2,\C$ or $\C^2$, denote
\begin{equation}\label{define of H(K)}
\mathcal{H}(\mathbb{K})=\{v\in H^{\infty}(\R,\mathbb{K}) \vert e^{\alpha |x|} \vert \partial_x^a v\vert \in L^{\infty}(\R) \text{ for any }a \in \N\}.
\end{equation}
We have the following properties of $L_{\C^2}$.
\begin{proposition}\label{proposition 19}
\begin{itemize}
\item[(i)] The eigenvalue $\lambda=\rho+i\theta$ can be chosen with maximal real part. We denote $Z(x)=\begin{pmatrix}Z^+(x)\\Z^-(x)\end{pmatrix} \in H^2(\R,\C)$, an associated eigenfunction.
\item[(ii)] $\phi\in\mathcal{H}(\R^2)$ and $Z\in\mathcal{H}(\C^2)$.
\item[(iii)] Let $\mu\notin Sp(L_{\R^2})$ and $A\in\mathcal{H}(\C^2)$. There exists a solution $X\in \mathcal{H}(\C^2)$ to $(L-\mu I)X=A$ and $(L-\mu I)^{-1}$ is a continuous operator on $\mathcal{H}(\C^2)$. 
\end{itemize}
\end{proposition}
Since $L_{\C^2}$ and $L_{\C^2}^{\tilde{\phi}}$ are conjugates of each other, we only need to prove Proposition \ref{proposition 19} for $L_{\C^2}^{\tilde{\phi}}$. 
\begin{proof}
(i) We see that if $\lambda$ is an eigenvalue of $L_{\C}$ with eigenfunction $v$ then $\lambda$ is an eigenvalue of $L_{\C^2}$ with eigenfunction $\begin{pmatrix}v\\ -iv \end{pmatrix}$. Thus, from \eqref{A1}, there exists an eigenvalue of $L_{\C^2}^{\tilde{\phi}}$ with positive real part. Since $L_{\C^2}^{\tilde{\phi}}$ is a compact perturbation of $\begin{pmatrix}0&\partial_{xx}-\frac{h_1^2}{4}\\-\partial_{xx}+\frac{h_1^2}{4}&0\end{pmatrix}$, the essential spectrum of $L_{\C^2}^{\tilde{\phi}}$ is the set $\left\{iy: y\in\R,|y|\geq \frac{h_1^2}{4}\right\}$ and there exists an eigenvalue $\lambda$ with maximal real part.\\
(ii) It is well known that $\phi$ and its derivative are exponentially decay with decay rate $\frac{h_1}{2}$. Combining with the fact that $\phi$ solves an elliptic equation, we have $\phi\in \mathcal{H}(\R^2)$. Since $L_{\C^2}^{\tilde{\phi}}Z=\lambda Z$, using Proposition \ref{pro25} (i), we have $Z\in \mathcal{H}(\C^2)$.\\
(iii) This part follows from Proposition \ref{pro25} (ii).   
\end{proof}

We need the following definition.
\begin{definition}
Let $\xi\in C^{\infty}(\R^+,H^{\infty}(\R))$ and $\chi:\R^+\rightarrow (0,\infty)$. Then we denote
\[
\xi(t)=O(\chi(t)) \quad \text{ as } t\rightarrow\infty,
\]
if, for all $s\geq 0$, there exists $C(s)>0$ such that 
\[
\forall t\geq 0, \quad \norm{\xi(t)}_{H^s}\leq C(s)\chi(t).
\]
\end{definition}
Define $Y_1:=\begin{pmatrix}\Re(Z)=\begin{pmatrix}
\Re(Z^+)\\ \Re(Z^-)
\end{pmatrix}\end{pmatrix}$ and $Y_2:=\Im(Z)=\begin{pmatrix}
\begin{pmatrix}
\Im(Z^+)\\ \Im(Z^-)
\end{pmatrix}
\end{pmatrix}$. Then $Y_1,Y_2\in \mathcal{H}(\R^2)$, and
\begin{equation}\label{properties Y_i}
\begin{cases}
L_{\R^2}Y_1=\rho Y_1-\theta Y_2,\\
L_{\R^2}Y_2=\theta Y_1+\rho Y_2.
\end{cases}
\end{equation}
Denote
\begin{equation}\label{define Y}
Y(t)=e^{-\rho t}(\cos(\theta t)Y_1+\sin(\theta t)Y_2).
\end{equation}
\begin{lemma}
The function $Y(t)$ solves the following equation.
\[
\partial_tY+L_{\R^2}Y=0.
\]
\end{lemma}
\begin{proof}
The desired result follows from \eqref{properties Y_i} and the definition of $Y$ \eqref{define Y}. For detail proof, we refer reader to \cite[Lemma 21]{CoLe11}.  
\end{proof}

\begin{proposition}\label{pro 22}
Let $N_0\in\N$ and $a\in\R$. Then there exists a profile $W^{N_0} \in C^{\infty}([0,\infty),\mathcal{H}(\R^2))$, such that 
\[
\partial_t W^{N_0}+L_{\R^2}W^{N_0}=\mathcal{M}_{\R^2}(W^{N_0})+O(e^{-\rho (N_0+1)t}),
\]
as $t\rightarrow\infty$ and $W^{N_0}(t)=aY(t)+O(e^{-2\rho t})$. 
\end{proposition}
For simplicity, in the proof of this proposition, we write $W$ for $W^{_0}$. We look for $W$ in the following form
\[
W(t,x)=\sum_{k=1}^{N_0}e^{-\rho k t}\left(A_{j,k}(x)\cos(j\theta t)+B_{j,k}(x)\sin(j\theta t)\right),
\]
where $A_{j,k}=\begin{pmatrix}A_{j,k}^+\\ A_{j,k}^-\end{pmatrix}$ and $B_{j,k}=\begin{pmatrix}B_{j,K}^+\\B_{j,k}^-\end{pmatrix}$ are some functions in $\mathcal{H}(\R^2)$ which are determined later. 
 
We have the following expression of $\mathcal{M}_{\R^2}(W)$.
\begin{lemma}
We have
\[
\mathcal{M}_{\R^2}(W)=\sum_{\kappa=2}^{N_0}e^{-\kappa \rho t}\sum_{j=0}^{\kappa} (\tilde{A_{j,\kappa}}(x) \cos(j\theta t)+\tilde{B_{j,\kappa}}(x)\sin(j\theta t))+O(e^{-(N_0+1)\rho t}),
\]
where $\tilde{A_{j,\kappa}}$, $\tilde{B_{j,\kappa}}$ depend on $A_{l,n},B_{l,n}$ and $\partial_xA_{l,n},\partial_xB_{l,n}$ only for $l\leq n\leq \kappa-1$.
\end{lemma}
\begin{proof}
Remark that there exists a polynomial $P_{N_0}\in \mathcal{H}(\R^2)[X,Y,Z,T]$ with coefficients in $\mathcal{H}(\R^2)$ and valuation at least $2$, such that
\begin{align*}
\mathcal{M}_{\R^2}(W)&=P_{N_0}(v^+,v^-,v_x^+,v_x^-)+O(|v|^{N_0+1})\\
&=\sum_{m=2}^{N_0}\sum_{p_1=0}^1\sum_{p_2=0}^1\sum_{j=0}^{m-p_1-p_2}\begin{pmatrix}P_{j,p_1,p_2,m}(x)v_+^j \partial_x v_{+}^{p_1}\partial_x v_{-}^{p_2} v_{-}^{m-j-p_1-p_2}\\Q_{j,p_1,p_2,m}(x)v_+^j \partial_x v_{+}^{p_1}\partial_x v_{-}^{p_2} v_{-}^{m-j-p_1-p_2} \end{pmatrix}+O(v^{N_0+1}).
\end{align*}
The rest of the proof follows from \cite[Claim 24]{CoLe11}.
\end{proof}

\begin{proof}[Proof of Proposition \ref{pro 22}]
The desired result is proved by similar argument in \cite[Proof of Proposition 22]{CoLe11}. 
\end{proof}
Define
\[
V_1^{N_0}(t,x):=e^{i\omega_1 t}W^{N_0}(t,x-c_1t), \quad U^{N_0}_1(t,x):=R_1(t,x)+V^{N_0}_1(t,x).
\]
Then we define
\begin{align}
Err_1^{N_0}(t,x)&:=i\partial_t U_1^{N_0}+\partial_{xx}U_1^{N_0}+i|U_1^{N_0}|^2\partial_xU_1^{N_0}+b|U_1^{N_0}|^4U_1^{N_0}\label{eqU_1^N_0}\\
&=i\partial_tV^1_{N_0}+\partial_{xx}V_1^{N_0}+i(|R_1(t)+V_1^{N_0}|^2\partial_x(R_1(t)+V_1^{N_0})-|R_1(t)|^2\partial_xR_1(t))\nonumber\\
&\quad +b(|R_1(t)+V_1^{N_0}|^4(R_1(t)+V_1^{N_0})-|R_1(t)|^4R_1(t))\nonumber\\
&=i(\partial_tV_1^{N_0}+\mathcal{L}_{\C}V_1^{N_0}+\mathcal{N}_{\C}(V_1^{N_0}))\nonumber.
\end{align}
Remarking that $V_1^{N_0}(t,x)=e^{i\omega_1 t}W^{N_0}(t,x-c_1 t)$ and $R_1(t,x)=e^{i\omega_1 t}\phi(x-c_1 t)$, we have
\begin{align*}
\partial_t V^1_{N_0}&=e^{i\omega_1 t}(i\omega_1 W^{N_0}+\partial_t W^{N_0}-c\partial_xW^{N_0})\\
\mathcal{L}_{\C}V_1^{N_0}&=e^{i\omega_1 t} \left(2\Re(\phi \overline{W^{N_0}})\partial_x\phi+|\phi|^2\partial_xW^{N_0}-i\partial_{xx}W^{N_0}-ib(|\phi|^4W^{N_0}+4|\phi|^2\phi\Re(\phi\overline{W_{N_0}}))\right) \\
\mathcal{N}_{\C}(V_1^{N_0})&=e^{i\omega_1 t}\mathcal{M}_{\C}(W^{N_0}).
\end{align*}
Thus,
\begin{align*}
Err_1^{N_0}(t,x)&=i(\partial_tV_1^{N_0}+\mathcal{L}_{\C}V_1^{N_0}+\mathcal{N}_{\C}(V_1^{N_0}))\nonumber\\
&=ie^{i\omega_1 t}(\partial_t W^{N_0}+L_{\C}W^{N_0}+\mathcal{M}_{\C}(W^{N_0})).
\end{align*}
By Proposition \ref{pro 22}, $Err_1^{N_0}(t,x)=O(e^{-\rho(N_0+1)t})$. Moreover, $W^{N_0}(t)=aY(t)+O(e^{-2\rho t})$ and then $V_1^{N_0}(t,x)=ae^{i\omega_1 t}Y(t,x-c_1t)+O(e^{-2\rho t})$, where $Y(t)$ is defined by \eqref{define Y}. This implies that, for all $s\geq 0$, there exists $C(N_0,s)$ such that
\begin{equation}\label{30}
\forall t\geq 0,\quad \norm{V_1^{N_0}}_{H^s}\leq C(N_0,s)e^{-\rho t}.
\end{equation}

\subsection{Proof of Theorems \ref{thm1} and \ref{thm2}}

\begin{proof}[Proof of Theorem \ref{thm1}]
Let $N_0$ to be determined later. Define
\begin{align*}
\varphi(t,x)&=\exp\left(\frac{i}{2}\int_{-\infty}^x|u(t,y)|^2\,dy\right)u(t,x)\\
\psi&=\exp\left(\frac{i}{2}\int_{-\infty}^x|u(t,y)|^2\,dy\right)\partial_x u(t,x) =\partial_x \varphi-\frac{i}{2}|\varphi|^2\varphi,\\
h(t,x)&=\exp\left(\frac{i}{2}\int_{-\infty}^x|U_1^{N_0}|\,dy\right)U_1^{N_0}(t,x),\\
k&=\exp\left(\frac{i}{2}\int_{-\infty}^x|U_1^{N_0}|\,dy\right)\partial_xU_1^{N_0}(t,x)=\partial_xh-\frac{i}{2}|h|^2h.
\end{align*}
From \cite[page 8]{Tinpaper3}, we see that if $u$ solves \eqref{eq 1} then $(\varphi,\psi)$ solves the following system
\begin{equation}\label{eq3}
\begin{cases}
L\varphi=P(\varphi,\psi),\\
L\psi=Q(\varphi,\psi),
\end{cases}
\end{equation}
where $L=i\partial_t+\partial_{xx}$ and 
\begin{align*}
P(\varphi,\psi)&=i\varphi^2\overline{\psi}-b|\varphi|^4\varphi,\\
Q(\varphi,\psi)&=-i\psi^2\overline{\varphi}-3b|\varphi|^4\psi-2b|\varphi|^2\varphi^2\overline{\psi}.
\end{align*}
From \eqref{eqU_1^N_0}, by similar arguments in \cite[page 9]{Tinpaper3}, we have $h,k$ solves the following system
\begin{equation}
\begin{cases}
Lh=P(h,k)+Err_1^{N_0}(1),\\
Lk=Q(h,k)+Err_1^{N_0}(2),
\end{cases}
\end{equation}
where
\begin{align*}
Err_1^{N_0}(1)&=Err_1^{N_0}\exp\left(\frac{i}{2}\int_{-\infty}^x|U_1^{N_0}|^2\,dy\right)-h\int_{-\infty}^x\Im(Err_1^{N_0}\overline{U_1^{N_0}})\,dy\\
Err_1^{N_0}(2)&=\partial_xErr_1^{N_0}(1)-i|h|^2Err_1^{N_0}(1)+\frac{i}{2}h^2\overline{Err_1^{N_0}(1)}.
\end{align*}
Since $Err_1^{N_0}=O(e^{-\rho(N_0+1)t})$, we have $(Err_1^{N_0}(1),Err_1^{N_0}(2))=O(e^{-\rho(N_0+1)t})$. We do a fixed point around $q:=(h,k)$ of \eqref{eq3}. Set $\tilde{w}:=(\tilde{\varphi},\tilde{\psi})=(\varphi,\psi)-(h,k)$, $F(\varphi,\psi)=(P(\varphi,\psi),Q(\varphi,\psi))$ and $\tilde{Err_1^{N_0}}=(Err_1^{N_0}(1),Err_1^{N_0}(2))=O(e^{-\rho(N_0+1)t})$. We have 
\begin{equation}\label{relation}
\tilde{\psi}=\partial_x\tilde{\varphi}-\frac{i}{2}(|\tilde{\varphi}+h|^2(\tilde{\varphi}+h)-|h|^2h).
\end{equation}
Moreover, $\tilde{w}$ solves the following system
\begin{align}\label{eq4}
L\tilde{w}&=F(\tilde{w}+q)-F(q)-\tilde{Err_1^{N_0}}.
\end{align}
In Duhamel form, $\tilde{w}$ satisfies, for $t\leq s$
\begin{equation*}
\tilde{w}(s)=S(s-t)w(t)-i\int_t^sS(s-\tau)(F(\tilde{w}+q)-F(q)-\tilde{Err_1^{N_0}})(\tau)\,d\tau.
\end{equation*}
Thus,
\[
S(-s)w(s)=S(-t)w(t)-i\int_t^sS(-\tau)(F(\tilde{w}+q)-F(q)-\tilde{Err_1^{N_0}})(\tau)\,d\tau.
\]
We find $\tilde{w}$ such that $\tilde{w}(t)\rightarrow 0$ as $t\rightarrow\infty$. Letting $s\rightarrow\infty$ as $\tilde{w}(s)\rightarrow 0$, we need to find $\tilde{w}$ satisfying the fixed point equation
\[
\tilde{w}(t)=i\int_t^{\infty}S(t-\tau)(F(\tilde{w}+q)-F(q)-\tilde{Err_1^{N_0}})(\tau)\,d\tau.
\]
We define the map
\[
\Phi: v \mapsto \Phi(v)=i\int_t^{\infty}S(t-\tau)(F(v+q)-F(q)-\tilde{Err_1^{N_0}})(\tau)\,d\tau.
\]
Let $B,T_0$ to be determined later. For $\tilde{w}\in C([T_0,\infty),H^2(\R)\times H^2(\R))$, define
\[
\norm{\tilde{w}}_{X_{T_0,N_0}}=\sup_{t\geq T_0}e^{\rho(N_0+1)t}\norm{\tilde{w}(t)}_{H^2\times H^2},\quad \text{ for } (\norm{\tilde{w(t)}}_{H^2\times H^2}=\norm{\tilde{\varphi}}_{H^2}+\norm{\tilde{\psi}}_{H^2})
\]
to be norm of the Banach space
\[
X_{T_0,N_0}:=\{\tilde{w}\in C((T_0,\infty),H^2(\R)\times H^2(\R))\vert \norm{\tilde{w}}_{X_{T_0,N_0}}<\infty\}.
\]
Define
\[
X_{T_0,N_0}(B):=\{\tilde{w}\in X_{T_0,N_0}\vert \norm{\tilde{w}}_{X_{T_0,N_0}} \leq B\}.
\]
We will find a fixed point of $\Phi$ in $X_{T_0,N_0}(B)$. By \eqref{30}, we can assume $T_0$ is large enough such that
\begin{align}\label{condition1}
Be^{-\rho(N_0+1)T_0}\leq 1,&\quad \text{ and } \norm{V_1^{N_0}}_{H^3}\leq 1.  
\end{align}
We see that 
\begin{align*}
\norm{q}_{H^2 \times H^2}&=\norm{h}_{H^2}+\norm{k}_{H^2}\\
&\leq C(\norm{U_1^{N_0}}_{H^3}+\norm{U_1^{N_0}}_{H^3}^3)\\
&\leq C(\norm{V_1^{N_0}}_{H^3}+\norm{V_1^{N_0}}_{H^3}^3+\norm{R_1}_{H_3}+\norm{R_1}_{H^3}^3)\\
&\leq C(2+\norm{\phi}_{H^3}+\norm{\phi}_{H^3}^3).
\end{align*}
Define $r=C(2+\norm{\phi}_{H^3}+\norm{\phi}_{H^3}^3)+1$. Due to smoothness of $F$, there exists a constant $K$ such that
\[
\forall a,b\in B_{H^2\times H^2}(r),\quad \norm{F(a)-F(b)}_{H^2\times H^2}\leq K\norm{a-b}_{H^2\times H^2}.
\]
In particular, 
\[
\norm{F(q+v)-F(q)}_{H^2\times H^2} \leq K\norm{v}_{H^2\times H^2}.
\]
For any $v\in X_{T_0,N_0}(B)$, we have
\begin{align*}
\norm{\Phi(v)}_{H^2\times H^2}&=\left\lVert \int_t^{\infty}S(t-\tau)(F(v+q)-F(q)-\tilde{Err_1^{N_0}})(\tau)\,d\tau\right\rVert_{H^2\times H^2}\\
&\leq \int_t^{\infty}(\norm{F(v+q)-F(q)}_{H^2\times H^2}+\norm{\tilde{Err_1^{N_0}}}_{H^2\times H^2})\,d\tau\\
&\leq \int_t^{\infty}(K\norm{v}_{H^2\times H^2}+C(N_0)e^{-\rho(N_0+1)\tau})\,d\tau\\
&\leq \frac{KB+C(N_0)}{(N_0+1)\rho}e^{-\rho(N_0+1)t}.
\end{align*}
Choose $N_0$ large enough such that $\frac{K}{(N_0+1)\rho} \leq \frac{1}{2}$ and choose $B=\frac{2C(N_0)}{(N_0+1)\rho}$. Finally, choose $T_0$ large enough such that \eqref{condition1} holds. Hence, we have
\[
\norm{\Phi(v)(t)}_{H^2\times H^2}\leq Be^{-\rho(N_0+1)t}.
\]
This implies that $\Phi$ maps $X_{T_0,N_0}(B)$ to itself. Now, we prove that $\Phi$ is a contraction in $X_{T_0,N_0}(B)$. Let $v_1,v_2\in X_{T_0,N_0}(B)$, we have
\begin{align*}
\Phi(v_1)(t)-\Phi(v_2)(t)&=i\int_t^{\infty}S(t-s)(F(v_1+q)-F(v_2+q))(s)\,ds.
\end{align*}
Thus,
\begin{align*}
&e^{\rho(N_0+1)t}\norm{\Phi(v_1)(t)-\Phi(v_2)(t)}_{H^2\times H^2}\\
&=e^{\rho(N_0+1)t}\left\lVert\int_t^{\infty}S(t-s)(F(v_1+q)-F(v_2+q))(s)\,ds\right\rVert_{H^2\times H^2}\\
&\leq e^{\rho(N_0+1)t}\int_t^{\infty}\norm{F(v_1+q)(s)-F(v_2+q)(s)}_{H^2\times H^2}\,ds\\
&\leq e^{\rho(N_0+1)t}\int_t^{\infty}K\norm{v_1-v_2}_{H^2\times H^2}\,ds\\
&\leq K e^{\rho(N_0+1)t}\int_t^{\infty}e^{-\rho(N_0+1)s}\norm{v_1-v_2}_{X_{T_0,N_0}}\,ds\\
&\leq K e^{\rho(N_0+1)t}\norm{v_1-v_2}_{X_{T_0,N_0}}\frac{e^{-\rho(N_0+1)t}}{(N_0+1)\rho}\\
&\leq \frac{K}{(N_0+1)\rho}\norm{v_1-v_2}_{X_{T_0,N_0}}.
\end{align*}
Taking supremum over $t\geq T_0$, we have
\[
\norm{\Phi(v_1)-\Phi(v_2)}_{X_{T_0,N_0}}\leq \frac{K}{(N_0+1)\rho}\norm{v_1-v_2}_{X_{T_0,N_0}}\leq \frac{1}{2}\norm{v_1-v_2}_{X_{T_0,N_0}}.
\]
Hence, $\Phi$ is a contraction on $X_{T_0,N_0}(B)$ and $\Phi$ has a fixed point $\tilde{w}$.

Next, we prove that the solution $\tilde{w}=(\tilde{\varphi},\tilde{\psi})$ of \eqref{eq4} satisfies the relation \eqref{relation} if $N_0$ is large enough. Define $v=\partial_x\varphi-\frac{i}{2}|\varphi|^2\varphi$ and $\tilde{v}=v-k=\partial_x\tilde{\varphi}-\frac{i}{2}(|\tilde{\varphi}+h|^2(\tilde{\varphi}+h)-|h|^2h)$. We need to prove that $\tilde{\psi}=\tilde{v}$. By similar argument as in \cite{Tinpaper3}, we have
\begin{align*}
L\tilde{\psi}-L\tilde{v}&=(\tilde{\psi}-\tilde{v})A(\tilde{\psi},\tilde{v},\tilde{\varphi},h,k)+\overline{\tilde{\psi}-\tilde{v}}B(\tilde{\psi},\tilde{v},\tilde{\varphi},h,k)-i(\tilde{\varphi}+h)^2\partial_x\overline{(\tilde{\psi}-\tilde{v})},
\end{align*}
where
\begin{align*}
A&=-i(\tilde{\psi}+\tilde{v}+2k)\overline{(\tilde{\varphi}+h)}-3b|\tilde{\varphi}+h|^4-\frac{1}{2}|\tilde{\varphi}+h|^4\\
B&=-2b|\tilde{\varphi}+h|^2(\tilde{\varphi}+h)^2-2i(\tilde{\varphi}+h)\left(\tilde{v}+k+\frac{i}{2}|\tilde{\varphi}+h|^2(\tilde{\varphi}+h)\right)-|\tilde{\varphi}+h|^2(\tilde{\varphi}+h)^2.
\end{align*}
Thus, 
\begin{align*}
&\norm{\tilde{\psi}(t)-\tilde{v}(t)}^2_{L^2}\\
&\lesssim \norm{\tilde{\psi}(N)-\tilde{v}(N)}^2_{L^2}\exp\left(\int_t^N(\norm{A}_{L^{\infty}}+\norm{B}_{L^{\infty}}+\norm{\partial_x(\tilde{\varphi}+h)^2}_{L^{\infty}})\,ds\right),\\
&\lesssim \norm{\tilde{\psi}(N)-\tilde{v}(N)}^2_{L^2} \exp\left((N-t)(\norm{A}_{L^{\infty}L^{\infty}}+\norm{B}_{L^{\infty}L^{\infty}}+...\right.\\
&\quad\quad\left.+2(\norm{\tilde{\varphi}}_{L^{\infty}}+\norm{h}_{L^{\infty}})(\norm{\partial_x\tilde{\varphi}}_{L^{\infty}}+\norm{\partial_x h}_{L^{\infty}}) )\right)\\
&\lesssim e^{-2\rho(N_0+1)N}e^{(N-t)C_*}, \quad \text{ for } N\gg t,
\end{align*}
where $C_*$ depends on $R_1$ (by using the bounded of $\norm{\tilde{\varphi}}_{H^2}+\norm{\tilde{\psi}}_{H^1}+\norm{h}_{H^2}+\norm{k}_{H^1}$). Choosing $N_0$ large enough and letting $N\rightarrow\infty$ we obtain $\tilde{\psi}=\tilde{v}$ and hence \eqref{relation} holds. Thus, we prove that there exists a solution $(\tilde{\varphi},\tilde{\psi})$ of \eqref{eq4} such that $\tilde{\psi}=\partial_x\tilde{\varphi}-\frac{i}{2}(|\tilde{\varphi}+h|^2(\tilde{\varphi}+h)-|h|^2h)$. Define $\varphi=\tilde{\varphi}+h$, $\psi=\tilde{\psi}+k$. Hence, $(\varphi,\psi)$ solves \eqref{eq3} and $\psi=\partial_x\varphi-\frac{i}{2}|\varphi|^2\varphi$. Setting 
\[
u(t,x)=\exp\left(-\frac{i}{2}\int_{-\infty}^x|\varphi(t,y)|^2\,dy\right)\varphi(t,x),
\]
we have $u$ solves \eqref{eq 1}. Moreover, 
\begin{align*}
\norm{u-U_1^{N_0}}_{H^2}&=\left\lVert\exp\left(\frac{-i}{2}\int_{-\infty}^x|\varphi(y)|^2\,dy\right)\varphi-\exp\left(\frac{-i}{2}\int_{-\infty}^x|h(y)|^2\,dy\right)h\right\rVert_{H^2}\\
&\lesssim \norm{\varphi-h}_{H^2}=\norm{\tilde{\varphi}}_{H^2}\leq Ce^{-\rho (N_0+1)t}, \quad \text{ for } t\geq T_0.
\end{align*}
Thus, $u(t)=R_1(t)+V_1^{N_0}(t)+O(e^{-2\rho t})$, for $t$ large enough. This completes the proof of Theorem \ref{thm1}. 
\end{proof}

\begin{proof}[Proof of Theorem \ref{thm2}]
Let $v_{\natural}$ to be fixed later and assume that $v_{*}>v_{\natural}$. Let $N_0$ to be defined later and $a\in\R$. Let $V_1^{N_0}(t)$, $U_1^{N_0}(t)$ and error term $Err_1^{N_0}(t)$ associated to $R_1(t)$ and an eigenvalue $\lambda=\rho+i\theta$ of $L_{\C}$. We look for a solution to \eqref{eq 1} of the form $u(t)=U_1^{N_0}(t)+\sum_{j\geq 2}R_j(t)+w(t)$. We use similar argument in the proof of Theorem \ref{thm1}. We define 
\begin{align*}
\varphi(t,x)&=\exp\left(\frac{i}{2}\int_{-\infty}^x|u(t,y)|^2\,dy\right)u(t,x),\\
\psi&=\partial_x\varphi-\frac{i}{2}|\varphi|^2\varphi,
\end{align*}
and
\begin{align*}
h(t,x)&=\exp\left(\frac{i}{2}\int_{-\infty}^x|U_1^{N_0}(t,y)+\sum_{j\geq 2}R_j(t,y)|^2\,dy\right)(U_1^{N_0}(t,x)+\sum_{j\geq 2}R_j(t,x)),\\
k&=\partial_xh-\frac{i}{2}|h|^2h.
\end{align*}
We see that if $u$ solves \eqref{eq 1} then $(\varphi,\psi)$ solves \eqref{eq3}.\\
Let $f(u)=i|u|^2u_x+b|u|^4u$ and $L$ be the Schr\"odinger operator defined as in the proof of Theorem \ref{thm1}. Define
\begin{align*}
Err_2^{N_0}:&=L(U_1^{N_0}+\sum_{j\geq 2}R_j)+f(U_1^{N_0}+\sum_{j\geq 2}R_j)
\end{align*}
Thus, by choosing $v_{\natural} \gg (N_0+1)\rho$ and Lemma \ref{lm1}, we have
\begin{align*}
Err_2^{N_0}&=LU_1^{N_0}+f(U_1^{N_0})+\sum_{j\geq 2}(LR_j+f(R_j))+(f(U_1^{N_0}+\sum_{j\geq 2}R_j)-f(U_1^{N_0})-\sum_{j\geq 2}f(R_j))\\
&=Err_1^{N_0}+(f(U_1^{N_0}+\sum_{j\geq 2}R_j)-f(U_1^{N_0})-\sum_{j\geq 2}f(R_j))\\
&=O(e^{-\rho(N_0+1)t})+O(e^{-h_* v_* t})=O(e^{-\rho (N_0+1)t}),
\end{align*}
Thus, by an elementary calculation, we have $q=(h,k)$ solve
\begin{align*}
Lq&=F(q)+\tilde{Err_2^{N_0}},
\end{align*}
where $F=(P,Q)$ is given as in the proof of Theorem \ref{thm1} and $\tilde{Err_2^{N_0}}=O(e^{-\rho(N_0+1)t})$.\\
Define $\tilde{w}=(\tilde{\varphi},\tilde{\psi})=(\varphi,\psi)-(h,k)$. Then $\tilde{w}$ solves
\begin{equation}
\label{eq15}
L\tilde{w}=F(\tilde{w}+q)-F(q)-\tilde{Err_2^{N_0}}.
\end{equation} 
By similar argument in the proof of Theorem \ref{thm1}, there exists a solution $\tilde{w}$ of \eqref{eq15} such that
\[
\sup_{t\geq T_0}e^{\rho(N_0+1)t}\norm{\tilde{w}(t)}_{H^2\times H^2} \leq B,
\]
for some $T_0,N_0,B$. From this and the Gr\"onwall inequality, we may prove that $\tilde{\psi}=\partial_x\tilde{\varphi}-\frac{i}{2}(|\tilde{\varphi}+h|^2(\tilde{\varphi}+h)-|h|^2h)$. Hence, we obtain a solution $u$ of \eqref{eq 1} such that
\begin{align*}
\norm{w(t)}_{H^2}&=\norm{u-U_1^{N_0}-\sum_{j\geq 2}R_j}_{H^2} \lesssim \norm{\varphi-h}_{H^2} =\norm{\tilde{\varphi}}_{H^2} \leq e^{-\rho(N_0+1) t},
\end{align*}
as $t$ large enough. Thus, $u(t)=U_1^{N_0}(t)+\sum_{j\geq 2}R_j(t)+ w(t)$ satisfies the desired property.
\end{proof}

\subsection{Orbital instability of soliton and multi-solitons}
In this section; we prove Corollary \ref{corollary1} and Corollary \ref{corollary2}.\\
Let $u\in C([T_0,\infty),H^2(\R))$ be the solution constrcuted in Theorem \ref{thm1}. Thus, 
\[
\forall t\geq T_0,\quad \norm{u(t)-R_1(t)-Y(t)}_{H^2}\leq Ce^{-2\rho t}.
\]
We have the following lemma. 
\begin{lemma}\label{lm2}
There exist $\varepsilon>0$, $t_0\geq T_0$ and $M\geq 0$ such that
\[
\inf_{y\in\R,\theta\in\R}\norm{u(t_0)-\phi(x-y)e^{i\theta}}_{L^2(B(0,M))}=\varepsilon>0.
\]
\end{lemma}

\begin{proof}
The proof of this lemma is similar the proof of \cite[Lemma 31]{CoLe11}.
\end{proof}

\begin{proof}[Proof of Corollary \ref{corollary1}]
Take a sequence $(S_n)$ such that $S_n\rightarrow\infty$ as $n\rightarrow\infty$, and define $T_n=t_0-S_n$ and 
\[
u_n(t,x)=u(t+S_n,x+c_1 S_n)e^{-i\omega_1 S_n}.
\]
Then $u_n\in C([T_n,0],H^2(\R))$ is a solution of \eqref{eq 1}. Since $u(t) \approx R(t)$ as $t \geq T_0$, we have
\[
u_n(t,x) \approx R_1(t+S_n,x+c_1S_n)e^{-i\omega_1 S_n}=e^{i\omega_1 t}\phi(x-c_1 t).
\] 
Thus,
\[
u_n(0,x)=\phi(x)+O(e^{-\rho S_n})
\]
and  hence
\[
\norm{u_n(0)-R_1(0)}_{H^2}\rightarrow 0\quad\text{ as } n\rightarrow\infty.
\]
Moreover,
\[
u_n(T_n,x)=u(t_0,x+c_1 S_n)e^{-i\omega_1 S_n}.
\]
Due to Lemma \ref{lm2}, we deduce that for all $n\in\N$, we have
\[
\inf_{y\in\R,\theta\in\R}\norm{u_n(T_n)-e^{i\theta}\phi(\cdot-y)}_{L^2}\geq \inf_{y\in\R,\theta\in\R}\norm{u(t_0)-e^{i\theta}\phi(\cdot-y)}_{L^2}\geq \varepsilon,
\]
which is the desired result.
\end{proof}

\begin{proof}[Proof of Corollary \ref{corollary2}]
Let $T>0$, $M$ be given by Lemma \ref{lm2} and $\varepsilon$, $(u_n)$, $(T_n)$ be given by Corollary \ref{corollary1}. Given $I<-T$, define $\tilde{u}_n\in C([I+T_n,I],H^2(\R))$ by
\[
\tilde{u}_n(t,x)=u_n(t-I,x-c_1t).
\]
By decreasing $I$ if possible, we assume that $\omega_1 I =0(2\pi)$. We have $\norm{\tilde{u}_n(I)-R_1(I)}_{H^2}=\norm{u_n(0)-R_1(0)}_{H^2} \rightarrow 0$, as $n\rightarrow\infty$ and $\tilde{u}_n(I+T_n)$ is $\varepsilon$-away from the $\phi$-soliton family. Consider the backward solution $w_n \in C((T^*,I],H^2(\R))$ of \eqref{eq 1} with the initial data at time $I$
\[
w_n(I,x)=\tilde{u}_n(I,x)+\sum_{j=2}^K R_j(I,x).
\]
If $T^*>-\infty$ then $w_n$ is a blow up solution. Consider the case $T^*=-\infty$. Note that $u_n \in C([T_n,0],H^2(\R))$ and $[0,T_n]$ is compact, the set $\{u_n(t)\vert t\in [0,T_n]\}$ is compact in $H^2(\R)$. Thus, $\sup_{t\in [0,T_n]} \norm{u_n(t)}_{H^2(|x|\geq R)}\rightarrow 0$ as $R\rightarrow\infty$. Hence, by the localized of $R_j$, there exists a function $\eta(I)$ such that $\eta(I)\rightarrow 0$ as $I\rightarrow -\infty$ and
\[
\forall t\in [I+T_n,I] \quad \sum_{j\geq 2}\norm{\tilde{u}_n(t)R_j(t)}_{H^2}\leq \eta(I).
\] 
Define $x_j(t)=c_jt+x_j$. Recall that $R_j(t,x)=e^{i\omega_j t}e^{i\theta_j}\phi_j(x-x_j(t))$. For $t<0$ small enough, $x_j(t)$ is far away from $x_1(t)$ for each $j\geq 2$. \\
Denote $J=I+T_n$ and
\[
z(t)=w_n(t)-\left(\tilde{u}_n(t)+\sum_{j=2}^K R_j(t)\right).
\]
Let $F=(P,Q)$ be given as in the proof of Theorem \ref{thm1}. Define
\begin{align*}
\varphi_n(t,x)&=w_n(t,x)\exp\left(\frac{i}{2}\int_{-\infty}^x|w_n(t,y)|^2\,dy\right),\\
\psi_n&=\partial_x\varphi_n-\frac{i}{2}|\varphi|^2\varphi_n,\\
h_n(t,x)&=(\tilde{u}_n(t,x)+\sum_{j=2}^K R_j(t,x))\exp\left(\frac{i}{2}\int_{-\infty}^x|\tilde{u}_n+\sum_{j=2}^K R_j|^2\,dy\right),\\
k_n&=\partial_xh_n-\frac{i}{2}|h_n|^2h_n,\\
\tilde{w}_n&=(\varphi_n,\psi_n)-(h_n,k_n),\\
q&=(h_n,k_n).
\end{align*}
Recall that $f(u)=i|u|^2u_x+b|u|^4u$. We have for $t\in [I+T_n,I]$
\begin{align*}
L(\tilde{u}_n+\sum_{j=2}^KR_j)+f(\tilde{u}_n+\sum_{j=2}^KR_j)&=f(\tilde{u}_n+\sum_{j=2}^KR_j)-f(\tilde{u}_n)-\sum_{j=2}^Kf(R_j)\\
&=\sum_{j\geq 2}O(\tilde{u}_n R_j)+\sum_{j\neq k\neq 1}O(R_jR_k)\leq C\eta(I), \text{ as } I\rightarrow -\infty.
\end{align*}
We see that $\tilde{w}_n(I)=0$. As in the proof of Theorem \ref{thm1}, we deduce that $\tilde{w}_n$ solves
\[
\tilde{w}_n=i\int_I^t S(t-s)(F(\tilde{w}_n)+q)-F(q)+Err)(s)\,ds,
\]
where $\norm{Err(s)}_{H^1\times H^1}\leq C\eta(I)$. Sine $F$ is lipschitz continuous on bounded set of $H^1(\R)\times H^1(\R)$, we have 
\begin{align*}
\norm{\tilde{w}_n(t)}_{H^1\times H^1}&\leq C\int_I^t (\norm{\tilde{w}_n(s)}_{H^1\times H^1}+\eta(I))\,ds\\
&\leq C\int_I^t\norm{\tilde{w}_n(s)}_{H^1\times H^1}\,ds+C\eta(I)(t-I).
\end{align*}
Hence, by Gr\"onwall inequality, we have
\[
\norm{\tilde{w}_n(t)}_{H^1\times H^1}\leq C\eta(I)(t-I)e^{C(t-I)}\leq C_n\eta(I),\quad \forall t\in [J,I].
\]
Thus, for $t\in [J,I]$
\begin{align*}
\norm{\varphi_n-h_n}_{H^2}&\lesssim \norm{\tilde{w}_n(t)}_{H^1\times H^1}\leq C_n\eta(I)
\end{align*}
Remark that $\tilde{u}_n(J)=u_n(T_n)$. This implies that for all $n$, we have
\[
\left\lVert w_n(J)-u_n(T_n)-\sum_{j=2}^K R_j(J)\right\rVert_{H^2}\lesssim \norm{\varphi_n-h_n}_{H^2}\leq C_n\eta(I).
\]
Choose $I_n$ such that $C_n \eta(I_n)<\frac{\varepsilon}{3}$, $J_n=I_n+T_n$. We have
\[
\norm{z(J_n)}_{H^2}\leq \frac{\varepsilon}{3}.
\]
Given $y_j,\gamma_j$, $c_j(t)=\omega_j t+\theta_j$ we have 
\begin{align*}
&\left\lVert w(J_n)-\sum_{j=1}^K\phi_j(\cdot-y_j)e^{i\gamma_j}\right\rVert_{L^2}\\
&\geq \left\lVert u_n(T_n)+\sum_{j=2}^KR_j(J_n)-\sum_{j=1}^K\phi_j(\cdot-y_j)e^{i\gamma_j}\right\rVert_{L^2}-\left\lVert w_n(J_n)-u_n(T_n)-\sum_{j=2}^KR_j(J_n)\right\rVert_{L^2}\\
&\geq \left\lVert u_n(T_n)-\phi(\cdot-y_1)e^{i\gamma_1}+\sum_{j=2}^K\phi_j(\cdot-x_j(J_n)e^{ic_j(J_n)}-\phi_j(\cdot-y_j)e^{i\gamma_j}\right\rVert_{L^2}-\frac{\varepsilon}{3}.
\end{align*}
If $\inf_{y_j,\gamma_j}\left\lVert w(J_n)-\sum_{j=1}^K\phi_j(\cdot-y_j)e^{i\gamma_j}\right\rVert_{L^2}> \varepsilon$ for infinite many $n$ then we obtain the desired result. We assume that for $n$ large enough, 
\[
\inf_{y_j,\gamma_j}\left\lVert w(J_n)-\sum_{j=1}^K\phi_j(\cdot-y_j)e^{i\gamma_j}\right\rVert_{L^2} \leq \varepsilon.
\]
Choosing $y_j,\gamma_j$ near minimizer such that 
\[
\left\lVert w_n(J_n)-\sum_{j=1}^K\phi_j(\cdot-y_j)e^{i\gamma_j}\right\rVert_{L^2} \leq 2\varepsilon.
\]
Consider $L^2$-norm on balls $B(x_j(J_n),R)$ around each $R_j$, $j\geq 2$. By localized of each $\phi_j$ and $u_n(T_n)=\tilde{u}_n(J_n)$, for $J_n$ small enough, we have
\begin{align*}
2\varepsilon+\varepsilon&\geq \left\lVert u_n(T_n)-\phi(\cdot-y_1)e^{i\gamma_1}+\phi_j(\cdot-x_j(J_n))e^{ic_j(J_n)}-\sum_{j=2}^K\phi_j(\cdot-y_j)e^{i\gamma_j}\right\rVert_{L^2(B(x_j(J_n),R))}\\
&\geq \left\lVert\phi_j(\cdot-x_j(J_n))e^{ic_j(J_n)}-\sum_{j=1}^K\phi_j(\cdot-y_j)e^{i\gamma_j}\right\rVert_{L^2(B(x_j(J_n),R))},\quad \forall j\geq 2.
\end{align*}
Thus, each $j\geq 2$ there exists $y_{k(j)}\neq 1$ near $x_j(J_n)$. Hence, each $j\geq 2$, there exists only one $y_{k(j)}$ near $x_j(J_n)$. Since $\phi_j \neq \phi_k$, for $j\neq k$ we have $k(j)=j$ for all $j\geq 2$ i.e $y_k-x_k(J_n)=O(1)$, for all $j \geq 2$ uniformly in $n$. This implies that
\begin{align*}
\left\lVert\sum_{j=2}^K\phi_j(\cdot-x_j(J_n))e^{ic_j(J_n)}-\phi_j(\cdot-y_j)e^{i\gamma_j}\right\rVert_{L^2} &=O_{I_n\rightarrow -\infty}(1)\leq \frac{\varepsilon}{3}. 
\end{align*} 
Thus, 
\begin{align*}
\inf_{y_j\in\R;\gamma_j\in\R}\left\lVert w_n(J_n)-\sum_{j=1}^K\phi_j(\cdot-y_j)e^{i\gamma_j}\right\rVert_{L^2}&\geq \left\lVert  w_n(J_n)-\sum_{j=1}^K\phi_j(\cdot-y_j)e^{i\gamma_j}\right\rVert_{L^2(B(0,M))}\\
&\geq \norm{u_n(T_n)-\phi(\cdot-y_1)e^{i\gamma_1}}_{L^2(B(0,M))}-\frac{2\varepsilon}{3}\\
&\geq \varepsilon-\frac{2\varepsilon}{3}=\frac{\varepsilon}{3},
\end{align*}
where we use Corollary \ref{corollary1}. Moreover,
\[
\left\lVert w_n(I_n)-\sum_{j=1}^KR_j(I_n)\right\rVert_{H^2}\rightarrow 0,
\]
as $n\rightarrow\infty$. Thus, we obtain the desired result.
\end{proof}

\section{Appendix}
In this section, we consider an operator $L: H^2(\R,\C^2)\subset L^2(\R,C^2)\rightarrow L^2(\R,\C^2)$ of the form
\[
L=\begin{pmatrix}
W_{1,1}+W_{1,2}\partial_x & \partial_{xx}-\frac{h_1^2}{4}+W_2\\
-\partial_{xx}+\frac{h_1^2}{4}+W_3& W_{4,1}+W_{4,2}\partial_x
\end{pmatrix},
\]
where $h_1\in\R$ and $W_{1,1},W_{1,2},W_2,W_3,W_{4,1},W_{4,2}$ belong to $\mathcal{H}(\C)$, where $\mathcal{H}(\C)$ is defined by \eqref{define of H(K)}. We prove the following result.
\begin{proposition}\label{pro25}
Let $\lambda\in \C\setminus \{iy,y\in\R,|y|\geq \frac{h_1^2}{4}\}$, and $U=\begin{pmatrix}u\\ v\end{pmatrix} \in H^2(\R,\C^2)$ such that $LU=\lambda U$. We have the following results.
\begin{itemize}
\item[(i)] There exist $C>0$ and $\alpha>0$ such that for all $x\in\R$ we have
\begin{equation}\label{need to solve}
|u(x)|+|v(x)|+|u'(x)|+|v'(x)|\leq Ce^{-\alpha |x|}.
\end{equation}
Moreover, $u,v\in\mathcal{H}(\C)$. 
\item[(ii)] Let $\lambda\notin Sp(L)$ and take $A\in\mathcal{H}(\C^2)$. Then there exists $X\in \mathcal{H}(\C^2)$ such that $(L-\lambda Id)X=A$.
\end{itemize}
\end{proposition}

To prove Proposition \ref{pro25}, we study the fundamental solutions to Helmholtz equations. For a given $\mu\in\C$, a fundamental solution to Helmholtz equation in $\R$ is a solution of 
\[
(-\partial_{xx}-\mu)g_{\mu}=\delta_0.
\]
For $\mu=\rho e^{i\theta}$ with $\rho\geq 0$ and $\theta\in (0,2\pi]$, we define $\sqrt{\mu}=\rho^{\frac{1}{2}}e^{\frac{i\theta}{2}}$. We have the following result (see \cite[Lemma 26]{CoLe11}).
\begin{lemma}\label{lemma bounded of fundamental solution}
Let $\mu\in \C\setminus \{\R^+\}$. Then there exist $\tau>0$ and $C>0$ such that
\[
|g_{\mu}(x)|\leq Cg_{-\tau}(x) \quad \forall x\in\R\setminus\{0\}.
\]
In particular, $g_{\mu}$ is exponentially decaying at infinity with decay rate $\sqrt{\tau}$ i.e $|g_{\mu}(x)|\leq Ce^{-\sqrt{\tau}|x|}$ for $|x|$ large enough. 
\end{lemma}

\begin{proof}
We have $\sqrt{\mu}=\rho^{\frac{1}{2}}e^{\frac{i\theta}{2}}$. It is well known that $g_{\mu}=\frac{i}{2\sqrt{\mu}}e^{i\sqrt{\mu}|x|}$. Thus, choosing $\tau>0$ such that $\sqrt{\tau}=\rho^{\frac{1}{2}}\sin\left(\frac{\theta}{2}\right)$, we have
\[
|g_{\mu}(x)|=\frac{1}{2|\sqrt{\mu}|}\vert e^{i\rho^{\frac{1}{2}}e^{i\frac{\theta}{2}}|x|}\vert \leq \frac{1}{2\sqrt{\rho}}e ^{-\rho^{\frac{1}{2}}\sin\left(\frac{\theta}{2}\right)|x|}.
\]  
Since
\[
g_{-\tau}(x)=\frac{1}{2\sqrt{\rho}\sin\left(\frac{\theta}{2}\right)}e ^{-\rho^{\frac{1}{2}}\sin\left(\frac{\theta}{2}\right)|x|},
\]
we obtain the desired result.
\end{proof}

The following regularity result on eigenfunctions is trivial.
\begin{lemma}\label{lemma 25}
Under the assumptions of Proposition \ref{pro25}, the functions $u,v \in H^{\infty}(\R,\C)$ and $\lim_{|x|\rightarrow\infty}(|u(x)|+|v(x)|+|\partial_xu(x)|+|\partial_xv(x)|)=0$.
\end{lemma}
For the rest of the proof, we work with the following operator
\[
L'=iPLP^{-1}=\begin{pmatrix}
\partial_{xx}-\frac{h_1^2}{4}+\tilde{W_{1,1}}\partial_x+\tilde{W_{1,2}}& \tilde{W_{2,1}}\partial_x+\tilde{W_{2,2}}\\
\tilde{W_{3,1}}\partial_x+\tilde{W_{3,2}}&-\partial_{xx}+\frac{h_1^2}{4}+\tilde{W_{4,1}}\partial_x+\tilde{W_{4,2}}
\end{pmatrix},
\]
where $P=\begin{pmatrix}
1 & i\\ 1&-i
\end{pmatrix}$ and
\begin{align*}
\tilde{W_{1,1}}&=\frac{i}{2}W_{1,2}+\frac{i}{2}W_{4,2}\\
\tilde{W_{1,2}}&=\frac{i}{2}W_{1,1}+\frac{1}{2}W_2-\frac{1}{2}W_3+\frac{i}{2}W_4\\
\tilde{W_{2,1}}&=\frac{i}{2}W_{2,1}-\frac{i}{2}W_{4,2}\\
\tilde{W_{2,2}}&=\frac{i}{2}W_{1,1}-\frac{i}{2}W_2-\frac{1}{2}W_3-\frac{i}{2}W_{4,1}\\
\tilde{W_{3,1}}&=\frac{i}{2}W_{1,2}-\frac{i}{2}W_{4,2}\\
\tilde{W_{3,2}}&=\frac{i}{2}W_{1,1}+\frac{1}{2}W_2+\frac{1}{2}W_3-\frac{i}{2}W_{4,1}\\
\tilde{W_{4,1}}&=\frac{i}{2}W_{1,2}+\frac{i}{2}W_{4,2}\\
\tilde{W_{4,2}}&=\frac{i}{2}W_{1,1}-\frac{1}{2}W_2+\frac{1}{2}W_3+\frac{i}{2}W_{4,1}.
\end{align*}
Thus, $\tilde{W_{i,j}} \in \mathcal{H}(\C)$ for each $i=1,...,4$ and $j=1,2$. Then the spectrum of $L'$ is $Sp(L')=iSp(L)$. We see that if $\lambda$ is an eigenvalue of $L$ with eigenvector $U$ then $\lambda'=i\lambda$ is an eigenvalue of $L'$ with eigenvector $U'=\begin{pmatrix}u'\\v'\end{pmatrix}=PU$. 

Write $L'-\lambda' I=H+K$, where
\begin{align*}
H:=\begin{pmatrix}
\partial_{xx}-\frac{h_1^2}{4}-\lambda' & 0\\
0& -\partial_{xx}+\frac{h_1^2}{4}-\lambda'
\end{pmatrix} & \quad\text{ and } K:=\begin{pmatrix}
\tilde{W_{1,1}}\partial_x+\tilde{W_{1,2}}& \tilde{W_{2,1}}\partial_x+\tilde{W_{2,2}}\\
\tilde{W_{3,1}}\partial_x+\tilde{W_{3,2}}& \tilde{W_{4,1}}\partial_x+\tilde{W_{4,2}}
\end{pmatrix}.
\end{align*}
Define
\[
F:=\begin{pmatrix}
f_1\\f_2
\end{pmatrix}:=KU'=\begin{pmatrix}
(\tilde{W_{1,1}}\partial_x+\tilde{W_{1,2}})u'+(\tilde{W_{2,1}}\partial_x+\tilde{W_{2,2}})v'\\
(\tilde{W_{3,1}}\partial_x+\tilde{W_{3,2}})u'+(\tilde{W_{4,1}}\partial_x+\tilde{W_{4,2}})v'
\end{pmatrix}.
\]
We have
\begin{align*}
u'=g_{-\frac{h_1^2}{4}-\lambda'}*(-f_1)&\quad v'=g_{\lambda'-\frac{h_1^2}{4}}*f_2.
\end{align*}
Let $\mu_1=-\frac{h_1^2}{4}-\lambda'$ and $\mu_2=\lambda'-\frac{h_1^2}{4}$. Since $\lambda \notin \left\{iy,y\in\R,|y|\geq\frac{h_1^2}{4}\right\}$, we have $\mu_1,\mu_2$ satisfy the assumption of Lemma \ref{lemma bounded of fundamental solution}. Let $\tau_1,\tau_2$ be given as in Lemma \ref{lemma bounded of fundamental solution} and set $\tau:=\min\{\tau_1,\tau_2\}$. Define
\begin{align*}
\tilde{F}:=\begin{pmatrix}
\tilde{f_1}\\\tilde{f_2}
\end{pmatrix}=\begin{pmatrix}
|f_1|\\|f_2|
\end{pmatrix} \quad \text{ and } \tilde{G}:=\begin{pmatrix}
\tilde{g_1}\\\tilde{g_2}
\end{pmatrix}=\begin{pmatrix}
|\partial_x f_1|\\|\partial_xf_2|
\end{pmatrix}
\end{align*}
\begin{align*}
\tilde{u}:=g_{-\tau}*\tilde{f_1} &\quad \text{ and }\tilde{v}=g_{-\tau}*\tilde{f_2}\\
\tilde{u^1}:=g_{-\tau}*\tilde{g_1} &\quad \text{ and }\tilde{v^1}=g_{-\tau}*\tilde{g_2}.
\end{align*}

\begin{lemma}\label{lemma 28}
There exists $C>0$ such that
\begin{align*}
|u'|\leq C\tilde{u} &\quad \text{ and } |v'|\leq C\tilde{v},
|\partial_x u'|\leq C\tilde{u^1}&\quad\text{ and } |\partial_xv'|\leq C\tilde{u^2}.
\end{align*}
\end{lemma}

\begin{proof}
From Lemma \ref{lemma bounded of fundamental solution}, $|g_{\mu_1}| \leq Cg_{-\tau_1}\leq Cg_{-\tau}$ for some $C>0$. Thus,
\begin{align*}
|u'|&= |g_{\mu_1}*(-f_1)|\leq Cg_{-\tau}*\tilde{f_1}=C\tilde{u},\\
|\partial_x u'|&=|g_{\mu_1}*\partial_x(-f_1)|\leq Cg_{-\tau}*\tilde{g_1}=C\tilde{u^1}.
\end{align*}
Similarly, we have $|v'|\leq C\tilde{v}$ and $|\partial_xv'|\leq C\tilde{v^1}$ for some $C>0$. This completes the proof.
\end{proof}

\begin{lemma}\label{lemma 29}
Set $w:=\tilde{u}+\tilde{v}+\tilde{u^1}+\tilde{v^1}$. There exist $C>0$ and $\alpha>0$ such that
\[
w(x) \leq Ce^{-\alpha|x|}, \quad \forall x\in\R.
\]
\end{lemma}
The proof of Lemma \ref{lemma 29} follows closely the proof of \cite[Theorem 1.1]{DeLi07} or \cite[Lemma 29]{CoLe11}.

\begin{proof}
Set $f:=\tilde{f_1}+\tilde{f_2}+\tilde{g_1}+\tilde{g_2}$. We have $w\in C^0(R)$. Indeed, $w$ solves
\begin{equation}\label{A9}
-\partial_{xx}w+\tau w=f,
\end{equation}
and from $f\in L^2(\R)$, this implies $w\in H^2(\R)$ and then $w \in C^0(\R)$.\\
Now, we prove that there exists $R>0$ such that for all $x\in\R$ with $|x|>R$ we have
\begin{align}\label{A10}
\frac{\tau w(x)-f(x)}{w(x)}&\geq \frac{\tau}{2}.
\end{align}
Indeed, setting $T(x):=\sum_{i=1}^4\sum_{j=1}^2|\tilde{W_{i,j}}|+|\partial_x\tilde{W_{i,j}}|$. Since $u'$ solves $(-\partial_{xx}-\mu_1)u'=f_1$, we have $|\partial_{xx}u'|\leq C(|u'|+|f_1|)\leq C(|u'|+|v'|+|\partial_x u'|+|\partial_xv'|)$, for some $C>0$. Similarly, $|\partial_{xx}v'|\leq C(|v'|+|f_2|)\leq C(|u'|+|v'|+|\partial_x u'|+|\partial_xv'|)$. Combining Lemma \ref{lemma 28}, we have
\begin{align*}
f&=\tilde{f_1}+\tilde{f_2}+\tilde{g_1}+\tilde{g_2}\\
&\leq T(x)(|u'|+|v'|+|\partial_x u'|+|\partial_xv'|+|\partial_{xx}u'|+|\partial_{xx}v'|)\\
&\leq CT(x)(|u'|+|v'|+|\partial_x u'|+|\partial_xv'|)=CT(x)w.
\end{align*}
Thus, 
\begin{align*}
\frac{\tau w(x)-f(x)}{w(x)}&\geq \tau -CT(x)\geq \frac{\tau}{2},
\end{align*}
for $|x|>R$ large enough, by decaying of the function $T$. This proves \eqref{A10}. 

Note that $w\geq 0$. Since $w\in C^0(\R)\cap H^2(\R)$, there exists $C_R$ such that for all $x\in\R$ with $|x|<R$, we have
\[
0\leq w\leq C_R.
\]
Define $\psi(x):=C_Re^{-\frac{\tau}{2}(|x|-R)}$. We have
\begin{align}
-\partial_{xx}\psi+\frac{\tau}{2}\psi\geq 0&\quad \text{ on } \R\setminus\{0\},\nonumber\\
 w(x)-\psi(x)\leq 0 &\quad \text{ on } \{x\in\R,|x|<R\}.\label{A11}
\end{align} 
Thus, we only need to prove that $w(x)\leq \psi(x)$ for $|x|>R$. We prove by contradiction. Assume that $w(x_0)>\psi(x_0)$ for some $|x_0|>R$. Define
\[
\Omega:=\{x\in\R,w(x)>\psi(x)\}.
\]
Then $\Omega$ is not empty and for all $x\in \Omega$, we have $|x|>R$ and for all $x\in\partial\Omega$ we have $w(x)=\psi(x)$. Moreover, by \eqref{A9}, \eqref{A10} and \eqref{A11}, we have 
\begin{align*}
\partial_{xx}(w-\psi)&=\partial_{xx}w-\partial_{xx}\psi=\tau w-f-\partial_{xx}\psi\\
&=\frac{\tau w-f}{w}w-\partial_{xx}\psi\geq \frac{\tau}{2}(w-\psi)>0.
\end{align*}
By maximal principle, this implies that $w-\psi\leq 0$ on $\Omega$, a contradiction. Thus, for all $x\in\R$ we have
\[
w(x)\leq \psi(x)=C_Re^{-\sqrt{\frac{\tau}{2}}(|x|-R)}=Ce^{-\sqrt{\frac{\tau}{2}}|x|}.
\] 
This implies the desired result.
\end{proof}

\begin{proof}[Proof of Proposition \ref{pro25}]
(i) By using Lemma \ref{lemma 25}, \ref{lemma 28} and \ref{lemma 29}, it is easy to imply that \eqref{need to solve} holds. Since $u,v$ solves a system of elliptic equations and \eqref{need to solve}, $u,v$ and their derivative are exponentially decaying at rate $\alpha$. This implies the desired result. \\
(ii) Since $\lambda\notin Sp(L)$, there exists $X\in H^2(\R,\C^2)$ such that $(L-\lambda Id)X=A$. Define $L'=iPLP^{-1}$, $X'=PX$, $\lambda'=i\lambda$ and $A'=iPA$ then
\[
(L'-\lambda' Id)X'=A'.
\]
Recall that $L'-\lambda'=H+K$. Set $Y=\begin{pmatrix}y_1\\y_2\end{pmatrix}:=KX'$, $A'=\begin{pmatrix}a_1\\a_2\end{pmatrix}$ and $X'=\begin{pmatrix}x_1\\x_2\end{pmatrix}$. We have
\[
x_1=g_{-\frac{h_1^2}{4}-\lambda'}*(y_1-a_1) \quad \text{ and }x_2=g_{\lambda'-\frac{h_1^2}{4}}*(a_2-y_2).
\] 
The terms $g_{-\frac{h_1^2}{4}-\lambda'}*(-a_1)$ and $g_{\lambda'-\frac{h_1^2}{4}} * a_2$ are exponentially decaying, with decay rate $\alpha$. Since $\tilde{W_{i,j}} \in \mathcal{H}(\C)$ for each $i=1,...,4$ and $j=1,2$, we have $y_1,y_2\in \mathcal{H}(\C)$. Hence, $g_{-\frac{h_1^2}{4}-\lambda'}*y_1$ and $g_{\lambda'-\frac{h_1^2}{4}} *y_2$ are exponentially decaying with decay rate $\alpha$. Moreover, for each multi-index $a$ we have $D^ax_k=g_{-\frac{h_1^2}{4}\pm \lambda'}*D^a(\pm (a_k-y_k))$, for $k=1,2$. This implies the decay of their derivatives of $X'$. This completes the proof of Proposition \ref{pro25}.  
\end{proof}

\begin{lemma}
\label{lm1}
Let $U_1^{N_0}$, $(R_j)$ ($j=1,...,K$) be profiles given as in the proof of Theorem \ref{thm2} and $f(u)=i|u|^2u_x+b|u|^4u$. Then
\[
f(U_1^{N_0}+\sum_{j\geq 2}R_j)-f(U_1^{N_0})-\sum_{j\geq 2}f(R_j)=O(e^{-h_* v_* |t|}),
\]
where $h_*$ and $v_*$ are defined as in Theorem \ref{thm2}.
\end{lemma}

\begin{proof}
For $j\neq k$, since \eqref{estimate R_j}, we have
\begin{align*}
|R_j(t,x)R_k(t,x)|&\lesssim e^{-\frac{h_j}{2}|x-c_jt|} e^{-\frac{h_k}{2}|x-c_kt|}\\
&\leq e^{-\frac{h_*}{2}(|x-c_jt|+|x-c_kt|)}\leq e^{-\frac{h_*}{2}|c_j-c_k| |t|} \leq e^{-3 h_* v_* |t|}.
\end{align*}
Thus, 
\begin{align*}
\norm{R_j(t)R_k(t)}_{L^2} &\leq \norm{\sqrt{|R_j(t)R_k(t)|}}_{L^{\infty}}\norm{\sqrt{|R_j(t)||R_k(t)|}}_{L^2}\leq  e^{-h_* v_* |t|},
\end{align*}
as $t$ large enough. By similar argument, we obtain the similar estimates of the interaction of the derivatives of $R_j$ and $R_k$.  Recall that $U_1^{N_0}=R_1+V_1^{N_0}$, where $V_1^{N_0}(t,x)=e^{i\omega_1 t}W^{N_0}(t,x-ct)$. Since, $W^{N_0} \in \mathcal{H}(\C)$, we have 
\begin{align*}
|R_j(t,x)V_1^{N_0}(t,x)|&\leq e^{-\frac{h_j}{2}|x-c_jt|}e^{-\alpha |x-c_1t|}\\
&\leq e^{-\frac{h_*}{2}|x-c_j t|}e^{-\frac{h_*}{2}|x-c_1t|}\leq e^{-3h_*v_* |t|}.
\end{align*}
Thus, we deduce that 
\[
\norm{R_j(t)V_1^{N_0}(t)}_{L^2}\leq e^{-h_*v_* |t|},
\]
for $t$ large enough. Similarly, we obtain the similar estimates of the interaction of the derivatives of $R_j$ ($j\geq 2$) and $V_1^{N_0}$. Moreover, we have
\[
f(U_1^{N_0}+\sum_{j\geq 2}R_j)-f(U_1^{N_0})-\sum_{j\geq 2}f(R_j)=\sum_{j\neq k\neq 1}O(R_jR_k)+\sum_{f\geq 2}O(V_1^{N_0}R_j).
\] 
This implies the desired result.
\end{proof}

\section*{Acknowledgement} I would like to thank Prof. Stefan Le Coz for his guidance and encouragement. I am supported by scholarship MESR in my phD. This work is also supported by the ANR LabEx CIMI (grant ANR-11-LABX-0040) within the French State Programme “Investissements d’Avenir.


\bibliographystyle{abbrv}
\bibliography{bibliothequepaper6}

\end{document}